\documentclass[10pt]{amsart}

\usepackage{amsmath, amsthm}
\usepackage{amssymb}
\usepackage{amscd}
\usepackage{latexsym}
\usepackage[latin1]{inputenc}
\usepackage[shortlabels]{enumitem}
\usepackage{extpfeil}

\usepackage{xypic}
\usepackage{graphicx}
\usepackage{changebar,color}

\usepackage{tikz-cd}

\newtheorem{theorem}{Theorem}[section]

\newtheorem{lemma}[theorem]{Lemma}

\newtheorem{proposition}[theorem]{Proposition}

\theoremstyle{remark}

\theoremstyle{definition}

\theoremstyle{definition}

\numberwithin{equation}{section}

\usetikzlibrary{decorations.markings}
\makeatletter
\tikzcdset{
open/.code={\tikzcdset{hook, circled};},
closed/.code={\tikzcdset{hook, slashed};},
open'/.code={\tikzcdset{hook', circled};},
closed'/.code={\tikzcdset{hook', slashed};},
circled/.code={\tikzcdset{markwith={\draw (0,0) circle (.375ex);}};},
slashed/.code={\tikzcdset{markwith={\draw[-] (-.4ex,-.4ex) -- (.4ex,.4ex);}};},
markwith/.code={
\pgfutil@ifundefined{tikz@library@decorations.markings@loaded}%
{\pgfutil@packageerror{tikz-cd}{You need to say %
\string\usetikzlibrary{decorations.markings} to use arrow with markings}{}}{}%
\pgfkeysalso{/tikz/postaction={/tikz/decorate,
/tikz/decoration={
markings,
mark = at position 0.5 with
{#1}}}}},
}
\makeatother

\renewcommand{\phi}{\varphi}

\def\A{\mathbb{A}}

\def\N{\mathbb{N}}
\def\Z{\mathbb{Z}}
\def\PP{\mathbb{P}}
\def\Q{\mathbb{Q}}

\def\C{\mathbb{C}}

\def\H{\mathbb{H}}

\newcommand{\cN}{{\mathcal N}}

\newcommand{\cO}{{\mathcal O}}

\newcommand{{\OL}}{{{\mathcal O}_L}}

\newcommand{\Proj}{{\rm Proj\,}}

\newcommand{\CH}{{\mathrm{CH}}}

\newcommand{\ra}{\rightarrow}
\newcommand{\lrasim}{\stackrel{\sim}{\longrightarrow}}
\newcommand{\lra}{\longrightarrow}

\newcommand{\hgt}{{\mathrm{ht}}}

\newcommand{\Dcirc}{{\mathaccent23 D}}

\renewcommand{\epsilon}{\varepsilon}

\newcommand{\GK}{\mathcal{GK}}

\title[Pencils of  hypersurfaces, Griffiths heights and geometric invariant theory. II]{Pencils of projective hypersurfaces, Griffiths heights and geometric invariant theory. II \\ Hypersurfaces with semihomogeneous singularities}

\author{Thomas Mordant}

\address{Universit\'e Paris-Saclay, Laboratoire de Math\'ematiques d'Orsay, 91405 Orsay Cedex, France}
\email{thomas.mordant@universite-paris-saclay.fr}

\begin{document}

\begin{abstract} This paper establishes a formula for the  stable Griffiths height of the middle-dimen\-sio\-nal cohomology of a pencil of projective hypersurfaces $H$, with semihomogeneous singularities,  over some smooth projective curve $C$. This formula appears as Theorem 5.1 in the first part of this paper \cite{MordantGIT1}. 

The proof of this formula relies on the strategy developed in our previous memoir \cite{Mordant22} to derive an expression for this Griffiths height when the only singularities of the fibers of $H$ over $C$ are ordinary double points. To deal with general semihomogeneous singularities, we complement this strategy by the construction of a finite covering $C'$ of $C$ such that the pencil $H'= H\times_C C'$ over $C'$ admits a smooth model $\widetilde{H}'$ with semistable fibers with smooth components. This allows us to circumvent the delicate issue of the determination of the elementary exponents attached to the singular fibers of~$H/C$. 
\end{abstract}

\maketitle

\tableofcontents

\section{Introduction}

\subsection{The stable Griffiths height of the middle-dimensional cohomology of a pencil of projective hypersurfaces with semihomogeneous singularities}

This paper is a sequel to \cite{MordantGIT1}, and is devoted to the proof of \cite[Theorem 5.1]{MordantGIT1}, which provides a closed formula for the stable Griffiths height of the middle-dimensional cohomology of a pencil of projective hypersurfaces with \emph{semihomogeneous singularities}.  

Let us recall the statement of this theorem:

\begin{theorem}
\label{pas intro GK hyp P(E) hom crit}
Let $C$ be a connected smooth projective complex curve with generic point $\eta$, and $E$ a vector bundle of rank~$N+1 \geq 2$ over $C$. Let $$\pi: \PP(E):= \Proj_C S^\bullet E^\vee \ra C$$ denote the attached projective bundle, and let $H \subset \PP(E) $ be a horizontal hypersurface of relative degree~$d\geq 2$, smooth over $\C$.

Let us assume that the set of critical points $\Sigma$ of the restriction 
$$f := \pi_{\mid H} : H \lra C$$
is finite, and that the restriction $$\pi_{| \Sigma} : \Sigma  \lra C$$  is injective. 
For every point $P$ in $H$, let $\delta_P$ be the multiplicity at $P$ of the projective hypersurface~$H_{\pi(P)}:= f^{-1}(\pi(P))$ in $\PP(E_{\pi(P)})$.\footnote{Observe that the positive integer $\delta_P$ is at least 2 if and only if $P$ is in $\Sigma.$}
Let us further assume that for every point $P$ of $\Sigma$, the projective tangent cone~$\PP(C_P H_{\pi(P)})$ of $H_{\pi(P)}$ at $P$ is smooth.

Then the following equality of integers holds:
\begin{equation} \label{f(sigma) P(E) hom crit}
\sum_{P \in \Sigma} (\delta_P - 1)^N = (N+1) (d-1)^N \, \hgt_{int}(H/C).
\end{equation}
Moreover the following equality of rational numbers holds:
\begin{equation} \label{pas intro GKstabXL P(E) hom crit}
\hgt_{GK, stab} \big(\H^{N-1}(H_\eta/C_\eta)\big) 
= - (N+1) w_{N,d} \; \hgt_{int}(H/C) + \sum_{P \in \Sigma} w_{N,\delta_P},
\end{equation}
where for every positive integer $\delta,$ $w_{N,\delta}$ is the rational number in $(1/12) \Z$ defined by:
\begin{equation}\label{wdef}
w_{N,\delta} = (\delta-1)\big[(N \delta + 1) (\delta - 1)^{N-1} + (-1)^N (\delta + 1) \big]/(12 \delta^2).
\end{equation}

\end{theorem}

The stable Griffiths height $\hgt_{GK, stab} \big(\H^{N-1}(H_\eta/C_\eta)\big)$ of the variation of Hodge structures $$\H^{N-1}(H_\eta/C_\eta)$$ attached to the middle-dimensional relative cohomology of  the pencil $H/C$ is defined in \cite[Section 2]{MordantGIT1}, and the  intersection-theoretic height $\hgt_{int}(H/C)$ of the hypersurface $H$ in the projective bundle~$\PP(E)$ in \cite[3.2.1]{MordantGIT1}.  

The requirement in Theorem \ref{pas intro GK hyp P(E) hom crit} on the  projective tangent cone~$\PP(C_P H_{\pi(P)})$ of $H_{\pi(P)}$ at any point $P$ of $\Sigma$ to be smooth precisely means that $P$ is a semihomogeneous singularity of $H_{\pi(P)}$. We refer the reader to \cite[Part II, 12 and 13]{AGZV85} for a general presentation and references concerning semihomogeneous (and more generally semiquasihomogeneous) singularities. 

In the special case where  the multiplicities $\delta_P$ of the points $P$ of $\Sigma$ are all  equal to 2 --- or equivalently when the singularities of the fibers of $H/C$ are ordinary double points --- then the equality \eqref{pas intro GKstabXL P(E) hom crit} in Theorem \ref{pas intro GK hyp P(E) hom crit} is easily seen, by means of  \eqref{f(sigma) P(E) hom crit}, to become the equality:
$$\mathrm{ht}_{GK,stab}\big(\H^{N-1}(H_\eta/C_\eta)\big)
= F_{stab}(d,N) \,  \mathrm{ht}_{int}(H/C),$$
established in \cite[Theorem 1.4.2]{Mordant22}, and used in the first part of this paper (see \cite[4.2]{MordantGIT1}).

\subsection{Main Theorem}

Theorem \ref{pas intro GK hyp P(E) hom crit} will be a consequence of a more general theorem, concerning hypersurfaces with semihomogeneous singularities in arbitrary smooth pencils. 
\subsubsection{Notation}\label{Not1} Consider a connected smooth projective complex curve $C$, with generic point $\eta$,  and:
$$\pi: X \lra C$$
 a smooth projective morphism of (necessarily smooth) complex projective varieties, whose fibers have pure    dimension $N.$ 
 
 Consider also a non-singular hypersurface  $H$ in $X$ such that the following conditions are satisfied:
\begin{enumerate}[(i)]
\item the set $\Sigma$ of critical points of the restriction $\pi_{\mid H}: H \ra C$ is finite;\footnote{This implies that $\pi_{\mid H}$ is a flat morphism.}

\item the restriction $\pi_{\mid  \Sigma} : \Sigma \ra C$ is injective;

\item for every point $P$ in $\Sigma,$ the projective tangent cone $\PP(C_P H_{\pi(P)})$  is non-singular. 
\end{enumerate}

We shall denote by:
$$\Delta := \pi(\Sigma)$$
the  ``locus of bad reduction" of $\pi_{\mid H}: H \ra C$.

For every $x \in \Delta$,  
we shall denote by $P_x$ the unique point in $\pi_{\mid \Sigma}^{-1}(x)$,  and by $\delta_{P_x}$ the multiplicity at~$P_x$  of the fiber $H_{x}.$  Observe that, inside the tangent space $T_{P_x} X$ of $X$ at the point $P_x$,  the hyperplanes $T_{P_x}X_x$ and $T_{P_x} H$ coincide, and that~$\delta_{P_x}$  is precisely the degree of the projective tangent cone $\PP(C_{P_x} H_{x})$ seen as a hypersurface in the projective space $\PP(T_{P_x} X_{x}) =  \PP(T_{P_x} H) \simeq \PP^{N-1}$.

We shall also denote by $L$ the line bundle $\cO_X(H)$ on $X.$

\subsubsection{}\label{112} The following theorem is the main result of this paper. 

\begin{theorem}\label{GK semi hom in X}
With the above notation, the following equality holds in $\CH_0(X)$:
\begin{equation}\label{sing hom hyp}
\sum_{P \in \Sigma} (\delta_{P}-1)^N \,[P] 
= [(1 - c_1(L))^{-1} c(\Omega^1_{X/C})]^{(N+1)}.
\end{equation}

Furthermore, if the line bundle $L$ on $X$ is ample relatively to the morphism $\pi,$ then the following equality of rational numbers holds:
\begin{multline}\label{eq: GK semi hom in X}
\hgt_{GK, stab}\big (\H^{N-1}(H_\eta / C_\eta)\big)
= \hgt_{GK}\big(\H^{N-1}(X/C)\big) - \hgt_{GK}\big(\H^N(X/C)\big) + \hgt_{GK}\big(\H^{N+1}(X/C)\big)\\+ 1/12 \int_X (1 - c_1(L))^{-1} c_1(\Omega^1_{X/C})\,  c(\Omega^1_{X/C})
- 1/12 \int_X c_1(L)\, c_N(\Omega^1_{X/C})
+ \sum_{P \in \Sigma} w_{N,\delta_{P}},
\end{multline}
where $w_{N, \delta}$ is defined by \eqref{wdef}.
\end{theorem}

In the right-hand sides of \eqref{sing hom hyp} and \eqref{eq: GK semi hom in X}, we denote by $c(\Omega^1_{X/C})$ the total Chern class of the vector bundle~$\Omega^1_{X/C}$.

This theorem extends our results  in \cite[6.1]{Mordant22}, where we consider  pencils of hypersurfaces whose only singularities are ordinary double points. Indeed  Proposition~6.1.1 in \emph{loc. cit.}, with  $\pi_{\mid \Sigma}$ assumed to be injective, is the variant concerning lower and upper  Griffiths heights of the special case of Theorem~\ref{GK semi hom in X} where all the $\delta_P$ are equal to 2.\footnote{The interested reader may actually check that the statement of   Theorem  \ref{GK semi hom in X} and its proof in the next sections remain valid when $\pi_{\mid \Sigma}$ is not assumed to be injective anymore, but when the multiplicity $\delta_P$ of a point $P$ in $\Sigma$ only depends on $\pi(P)$. This covers the variant of Proposition ~6.1.1 in \emph{loc. cit.} concerning stable Griffiths heights.}

\subsubsection{}\label{123} The derivation of Theorem \ref{pas intro GK hyp P(E) hom crit} from Theorem \ref{GK semi hom in X} will be done in Section \ref{Proof PE semihom}.

The proof of Theorem \ref{GK semi hom in X} will be completed at the end of Section \ref{geom hyp smooth pencil}, devoted to the geometry of the pencil of hypersurfaces $H$ in the smooth pencil $X$ over $C$.  It will rely on the auxiliary results from \cite{Mordant22} gathered in Section \ref{FromMordant22}, on the construction of a semistable model of $H$ over some finite covering $C'$ of the curve $C$ developed in Section \ref{ConstrSemstable}, and on the computations of characteristic classes of K\"ahler and logarithmic relative differentials in Section \ref{CharRelDiff}, which themselves are minor variations on computations in \cite{Mordant22}. 

Section \ref{Proof both omega} provides the proof of a technical result stated in Section \ref{CharRelDiff} (Proposition \ref{both omega iso}). It follows from standard arguments, but considering the length of its derivation, we preferred to give some details. 

\subsection{Discussion} Let us briefly compare the content of this paper with that of our previous works \cite{Mordant22} and~\cite{MordantGIT1}.

First of all, this paper deals with pencils of hypersurfaces that admit arbitrary semihomogeneous singularities, while our treatment in \cite{Mordant22} was restricted to the generic situation where the only allowed singularities of the fibers are ordinary double points. This gain in generality has been crucial in the first part \cite{MordantGIT1} of this paper to work out  sensible conjectural inequalities relating the stable Griffiths height and the GIT height of pencils of hypersurfaces with possibly arbitrary singular fibers. However the proofs in this paper 
do not rely on results established in \cite{MordantGIT1}.

A major difference between the proof of Theorem \ref{GK semi hom in X} and the proof of its ``special case" Proposition 6.1.1 in \cite{Mordant22} is the introduction in Section  \ref{ConstrSemstable} of the covering $C'$ of $C$ such that, base changed to $C'$, the pencil of hypersurfaces $H$ over $C$ admits a model $\widetilde{H}'$ with semistable fibers with smooth components. This construction allows us to compute directly the stable Griffiths height~$\hgt_{GK, stab}\big(\H^{N-1}(H_\eta/C_\eta)\big)$ with no knowledge of the elementary exponents attached to the singular fibers of $H/C$ (compare \cite[2.2-5]{Mordant22}).

In view of the subtleties involved in the computation of elementary exponents, even in the simplest case of ordinary double points dealt with in \cite[Chapter 2]{Mordant22}, 
this apparently technical device constitutes, we believe, an important new tool for the computation of stable Griffiths heights. It should also be  useful in the (yet to be completed) computation of   their arithmetic counterparts, the stable  Kato heights (see  \cite{Kato18}, and \cite[1.1.1]{Mordant22} for a discussion and additional references).

\subsection{Acknowledgements} I am very grateful to  Isabelle Mordant for her remarks on the writing of this paper, and to  Damien Simon for his \LaTeX-nical assistance during the preparation of this paper. 
This work constitutes a sequel of my PhD memoir \cite{Mordant22}, written under the supervision of Jean-Beno\^it Bost. I thank him for discussions on the themes explored in this paper, and for his advice during its completion. 

I am grateful to an anonymous reviewer for a careful reading of a first version of this paper, and for an extremely helpful list of corrections and suggestions of improvements of  its exposition --- notably for pointing out a gap in the original proof of Proposition \ref{both omega iso}.

During the preparation of this paper, I benefitted from the support of the Fondation Math\'e\-ma\-ti\-que Jacques Hadamard, as a \emph{Hadamard Lecturer}.

\section{Proof of Theorem \ref{pas intro GK hyp P(E) hom crit} from Theorem \ref{GK semi hom in X}}\label{Proof PE semihom} 

Before turning to the proof of Theorem \ref{GK semi hom in X}, we explain  how it implies Theorem \ref{pas intro GK hyp P(E) hom crit}.

We adopt the notation introduced in Theorem \ref{pas intro GK hyp P(E) hom crit}, and we introduce some further notation, similar to the one in \cite[6.2.1]{Mordant22}. 

\subsection{Relations between cycle classes in $\CH_0(\PP(E))$ associated to the hypersurface $H$} The structure of the Chow groups of $\PP(E)$ (see \cite[Th. 3.3 (b)]{Fulton98}, with $k = \dim(\PP(E)) - 1$), implies that the line bundle $L := \cO_{\PP(E)}(H)$ can be written:
$$L \simeq \cO_E(d) \otimes \pi^* M,$$
where $M$ is some line bundle on $C.$

Moreover we define  classes in $\CH^1(C)$ and $\CH^1(\PP(E))$ by:
$$e := c_1(E), \quad m := c_1(M),  \quad \mbox{and} \quad h := c_1(\cO_E(1)).$$
With this notation, the height $\hgt_{int}(H/C)$ is given by (see \cite[(6.2.13)]{Mordant22}):
\begin{equation} \label{ht int semi hom} \hgt_{int}(H/C) = \int_C \big(m - d/(N+1) \; e \big).
\end{equation}

Moreover, according to  the definition of the Segre classes of $E$ and to their relation to the Chern classes of $E$ (see for instance \cite[3.1, 3.2]{Fulton98}), the following equalities hold in $\CH^0(C)$ and $\CH^1(C)$ respectively:
\begin{equation} \label{segre semihom} \pi_* h^N = [C] \quad \mbox{and} \quad  \pi_* h^{N+1} = -e.
\end{equation}

The proof of Theorem \ref{pas intro GK hyp P(E) hom crit} shall rely on the following generalization of \cite[Proposition 6.2.5]{Mordant22}.
\begin{proposition}
\label{cycles Proj E semihom} 
With the above notation, the following equalities hold in $\CH_0(\PP(E))$:
\begin{equation}\label{sigma in PE semihom}
\sum_{P \in \Sigma} (\delta_P-1)^N [P] = (d-1)^N h^N [(d-1) h + (N+1) \pi^\ast m - \pi^\ast e], \end{equation}
\begin{equation}\label{c1 cN Omega in PE semihom}
c_1(L) c_N(\Omega^1_{\PP(E)/C}) = (-1)^N h^N [d (N+1) h + (N+1) \pi^\ast m + d N \pi^\ast e], 
\end{equation}
and:
\begin{equation}\label{quotient in PE semihom}
    [(1 - c_1(L))^{-1} c_1(\Omega^1_{\PP(E)/C}) c(\Omega^1_{\PP(E)/C}) ]^{(N+1)} = h^N (a_{N,d} h + b_{N,d} \pi^\ast m + c_{N,d} \pi^\ast e), 
\end{equation}
where $a_{N,d}$, $b_{N,d}$ and $c_{N,d}$ are the integers defined by:
$$a_{N,d} := \frac{N+1}{d} \big (- (d-1)^{N+1} + (-1)^{N+1} \big ),$$
$$b_{N,d} := \frac{N+1}{d^2} \big (- (d-1)^N (d N + 1) + (-1)^N \big ),$$
and:
$$c_{N,d} := \frac{1}{d}  \big ( - (d-1)^N (d-N-2) + (-1)^{N+1} (N+2) \big ).  $$
\end{proposition}

\begin{proof}
The proof follows from the  computations of Chern classes on $H$ and $\PP(E)$  in the proof of \cite[Proposition 6.2.5]{Mordant22}, by replacing the 0-cycle $i_* [\Sigma]$ in \emph{loc. cit.} by $\sum_{P \in \Sigma} (\delta_P-1)^N [P]$ and using the expression of this sum given by equality~\eqref{sing hom hyp}, applied to the hypersurface $H$ in the smooth pencil $X:=\PP(E)$ over $C$.
\end{proof}

\subsection{Computation of the stable Griffiths height} As in the proof of \cite[Theorem 6.2.4]{Mordant22}, pushing forward equality \eqref{sigma in PE semihom} by the morphism~$\pi,$ then using equalities \eqref{segre semihom}, and finally taking the degree and using equality \eqref{ht int semi hom} establishes equality~\eqref{f(sigma) P(E) hom crit}.

Reasoning similarly with equalities \eqref{c1 cN Omega in PE semihom} and \eqref{quotient in PE semihom} yields the following equalities:
\begin{equation}\label{push c1 cN Omega in PE semihom}
\int_{\PP(E)} c_1(L) c_N(\Omega^1_{\PP(E)/C}) = (-1)^N (N+1)\, \hgt_{int}(H/C)
\end{equation}
and:
\begin{align}
\int_{\PP(E)} (1 - c_1(L))^{-1} c_1(\Omega^1_{\PP(E)/C}) c(\Omega^1_{\PP(E)/C})
& = \int_C (b_{N,d} m + (c_{N,d} - a_{N,d}) e) \nonumber \\
\label{push quotient in PE semihom}
& = (N+1)/d^2 \;  (- (N d + 1) (d-1)^N + (-1)^N)\,  \hgt_{int} (H/C).
\end{align}

Applying equality \eqref{eq: GK semi hom in X} from Theorem \ref{GK semi hom in X} to the hypersurface $H$ in the smooth pencil $X := \PP(E)$ over $C$ (indeed the line bundle $L := \cO_{\PP(E)}(H)$ is relatively ample over $C$), and using that the Griffiths height of the relative cohomology in any degree of $\PP(E)$ over $C$ vanishes, we obtain the following equality:
\begin{multline}
\hgt_{GK, stab}\big(\H^{N-1}(H_\eta / C_\eta)\big)
= 1/12\;  \int_{\PP(E)} (1 - c_1(L))^{-1} c_1(\Omega^1_{{\PP(E)}/C})\,  c(\Omega^1_{{\PP(E)}/C}) \\
- 1/12 \; \int_{\PP(E)} c_1(L)\, c_N(\Omega^1_{{\PP(E)}/C})
+ \sum_{P \in \Sigma} w_{N,\delta_{P}}.
\end{multline}
 
 Combining this equality with equalities \eqref{push c1 cN Omega in PE semihom} and \eqref{push quotient in PE semihom} yields the following equalities:
\begin{align*}
\hgt_{GK, stab}\big(\H^{N-1}(H_\eta / C_\eta)\big)
& = \bigg[\frac{N+1}{12 d^2} (- (N d + 1) (d-1)^N + (-1)^N) - \frac{(-1)^N}{12} (N+1)\bigg] \, \hgt_{int}(H/C) \\
 & \quad \quad \quad +  \sum_{P \in \Sigma} w_{N,\delta_{P}}\\
& = - (N+1) w_{N,d}\,  \hgt_{int}(H/C) +  \sum_{P \in \Sigma} w_{N,\delta_{P}}.
\end{align*}
This concludes the proof of Theorem \ref{pas intro GK hyp P(E) hom crit}.

\section{Some results from \cite{Mordant22}}\label{FromMordant22}

The remainder of this paper is devoted to the  proof of Theorem \ref{GK semi hom in X}. It will rely on various auxiliary results established in \cite{Mordant22}, which we recall in this section.

\subsection{Alternating sums of Griffiths heights} Firstly it will use the expression for the alternating sum of Griffiths heights associated to a pencil of projective varieties whose singular fibers are divisors with strict normal crossings established in \cite[Theorem 4.2.3]{Mordant22}.

For the convenience of the reader, we recall this expression.

Let $C$ be a connected smooth projective complex curve with generic point $\eta$, $Y$ be a connected smooth projective $N$-dimensional complex scheme, and let
$$
g : Y \lra C
$$
be a surjective morphism of complex schemes. Let us assume that there exists a finite subset $\Delta$ in $C$ such that $g$ is smooth over $C \setminus \Delta$, and such that the divisor $D:=Y_\Delta$ is a divisor with strict normal crossings in $Y$.

We denote by
$$\omega^1_{Y/C} = \Omega^1_{Y/C} (\log Y_\Delta)$$
the vector bundle over $Y$ of relative logarithmic differentials.

We write the divisor $D:= Y_\Delta$ as:
$$
D = \sum_{i \in I} m_i D_i,
$$
where $I$ is a finite set, for every $i$ in $I$, $m_i \geq  1$ is an integer, and $D_i$ is a smooth connected divisor, such that the $(D_i)_{i \in I}$ intersect each other transversally. 

The set $I$ may be written as the disjoint union:
$$I = \bigcup_{x \in \Delta} I_x,$$
where, for every $x \in \Delta,$ $I_x$ denotes the non-empty subset of $I$ defined by:
$$I_x := \big\{ i \in I \mid g(D_i) = \{x \} \big\}.$$

For every subset $J$ of $I$, let us denote:
$$
D_J := \bigcap_{i \in J} D_i.
$$
It is a smooth subscheme of $Y$ of codimension $|J|$.
As in  \cite[II, 3.4]{Deligne70},  for every integer $r \geq 1$, we denote by $D^r$ the subscheme of codimension $r$ of $Y$ defined by the union of all the intersections of~$r$ different components $(D_i)$:
$$
D^r := \bigcup_{J \subset I, |J| = r} D_J.
$$

Let us choose a total order $\preceq$ on $I$.
For every element $i$ in $I$, we define an open subset $\Dcirc_i$ of the divisor $D_i$ by:
$$\Dcirc_i := D_i \setminus D_i \cap D^2.$$
Similarly, for every pair $(i,j)$ in $I^2$ such that $i \prec j$, we define an open subset $\Dcirc_{ij}$ of the subscheme
$$D_{ij} := D_{\{i,j\}}$$
by:
$$\Dcirc_{ij} := D_{ij} \setminus D_{ij} \cap D^3.$$

Finally we denote by $\chi_{\mathrm{top}}$  the topological Euler characteristic.

\begin{theorem}[{\cite[Theorem 4.2.3]{Mordant22}}]
\label{GK in terms of classes in D non-reduced}
With the above notation,  we have: 
\begin{equation}\label{eq GK in terms of classes in D non-reduced}
    \sum_{n=0}^{2(N-1)} (-1)^{n-1} \hgt_{GK, -}\big(\H^n(Y_\eta/C_\eta)\big) = \frac{1}{12} \int_Y (c_1 c_{N-1})(\omega^{1 \vee}_{Y/C}) + \sum_{x \in \Delta} \alpha_x,
\end{equation}
where for every $x$ in $\Delta$, $\alpha_x$ is the rational number given by:
\begin{equation}
\alpha_x = \frac{N-1}{4} \sum_{i \in I_x} (m_i - 1) \chi_{\mathrm{top}}(\Dcirc_i) + \frac{1}{12} \sum_{\substack{(i,j) \in I_x^2, \\ i \prec j}} (3 - m_i/m_j - m_j/m_i) \, \chi_{\mathrm{top}}(\Dcirc_{ij}). \label{geommult}
\end{equation}
\end{theorem}

In the proof of Theorem \ref{GK semi hom in X}, we shall only use the case of Theorem \ref{GK in terms of classes in D non-reduced} where the divisor $Y_{\Delta}$ is reduced, i.e. the multiplicities $m_i$ are all equal to 1. In this case equality \eqref{eq GK in terms of classes in D non-reduced} becomes the following equality (also given in \cite[(1.3.8)]{Mordant22}):
\begin{equation}\label{eq GK in terms of classes in D reduced semihom}
  \sum_{n=0}^{2(N-1)} (-1)^{n-1} \hgt_{GK}\big(\H^n(Y_\eta / C_\eta)\big)
= \frac{1}{12} \int_Y c_1(\omega_{Y/C}^{1\vee}) c_{N-1}(\omega_{Y/C}^{1\vee}) +  \frac{1}{12} \chi_{\mathrm{top}}(D^2 \setminus D^3).
\end{equation}

\subsection{A combinatorial formula}  The following combinatorial  formula will also be used in the proof of Theorem~\ref{GK semi hom in X}. 

\begin{proposition}[{\cite[Proposition 5.3.1]{Mordant22}}]
\label{formal identities semihom}
For every $a$ in $\C^\ast$, and every $n$ and $r$ in $\N$, the following equalities hold:
\begin{align}
\left[ \frac{(1 + y)^n}{1 + a y} \right]^{[n-r]} & = (-1)^{r} \sum_{k=r}^n \binom{n}{k} (-1)^k a^{k-r} = \frac{(-1)^{n+r}}{a^r} \Big [(a-1)^n - \sum_{k=0}^{r-1} \binom{n}{k} (-1)^{n-k} a^k \Big ],
\label{non square denominator semihom} 
\end{align}
where $f(y)^{[p]}$ denotes the coefficient of $y^p$ in some formal series $f(y) \in \C[[y]]$.\footnote{In particular, it vanishes when $p$ is negative.}
\end{proposition}

\section{Construction of a semistable model of $H$}\label{ConstrSemstable}  

In this section, we introduce  an auxiliary construction which will allow us to compute the stable Griffiths height in the left-hand side of \eqref{eq: GK semi hom in X} as the   Griffiths height attached to some family of projective varieties with semistable reduction, hence with unipotent monodromy.

\subsection{A local description of the divisor $H$ in $X$}\label{local desc}

Let $P$ be a point in $\Sigma,$ 
and let $t$ be a local coordinate of $C$ near $\pi(P)$, namely a uniformizer of the discrete valuation ring $\cO_{C, \pi(P)}$. 

The pullback $\pi^* t$ can be completed into a local coordinate system $(\pi^* t, x_1,\dots,x_N)$ of $X$ near~$P$.  In other words, $x_1, \dots, x_N$ belong to the local ring $\cO_{X,P}$ and  $(\pi^* t, x_1,\dots,x_N)$ generate its maximal ideal $\mathfrak{m}_{X,P}$. Since $P$ belongs to $\Sigma$, the hyperplanes $T_{P}X_{\pi(P)}$ and $T_{P} H$ in $T_{P} X$ coincide, and therefore the restrictions $(x_{1 | H},\dots,x_{N | H})$ induce a local coordinate system on $H$ near $P$.

Let $F_{\delta_P} (T_1, \dots, T_N)$  be an equation in $\C[T_1, \dots, T_N]$ of the projective tangent cone $\PP(C_P H_{\pi(P)})$ seen as a projective hypersurface in $\PP(T_P H)$ that we identify with $\PP^{N-1}_\C$ through the (differential at~$P$ of the) coordinate system $(x_{1 | H},\dots,x_{N | H})$. Since by hypothesis $\PP(C_P H_{\pi(P)})$  is a smooth hypersurface of degree $\delta_P$, the homogeneous  polynomial $F_{\delta_P}$ has degree $\delta_P$, and has non-zero discriminant. 

\begin{proposition}\label{Local H}
After possibly replacing $F_{\delta_P}$ by $\lambda F_{\delta_P}$ for some $\lambda$ in~$\C^\ast$, the germ at $P$ of the divisor $H$ in $X$ is 
 defined by an equation of the form:
\begin{equation} \label{eq H in X F delta local}
 \pi^* t = F_{\delta_P}(x_1,\dots,x_N) + F_{> \delta_P},
 \end{equation}
where $F_{>\delta_P}$ belongs to the ideal $\mathfrak{m}_{X,P}^{\delta_P +1}$ of $\cO_{X,P}$.
\end{proposition}

\begin{proof}
The divisor $H_{\pi(P)}$ in $H$ is defined by the equation $(\pi_{| H}^* t = 0)$. 
Therefore there exist $\lambda \in \C^\ast$ and ${G}_{> {\delta_P}} \in \mathfrak{m}_{H,P}^{{\delta_P} +1}$ such that the following 
equality holds in $\cO_{H,P}$:
$$\pi_{\mid H}^\ast t =\lambda F_{\delta_P} (x_{1 | H},\dots,x_{N | H}) + G_{> \delta_P}.$$
We may assume $\lambda = 1$ and  choose  an element $F_{>{\delta_P}}$ of $\mathfrak{m}_{X,P}^{{\delta_P} +1}$ whose restriction to (the germ at $P$ of) $H$ is ${G}_{> {\delta_P}}$. Then the equation \eqref{eq H in X F delta local} defines a germ of smooth hypersurface in $X$ which contains, hence coincides with, the germ of $H$ at $P$.
 \end{proof}

\subsection{Construction of the blow-ups $\widetilde{H}$ and $\widetilde{H}'$}
\subsubsection{}\label{641} Let $C'$ be a connected smooth projective complex curve, and let
$$\sigma : C' \lra C$$ be a finite morphism such that, for every point $x'$ in the finite set
$$\Delta' := \sigma^{-1}(\Delta),$$
the index of ramification of $\sigma$ at $x'$  is precisely $\delta_{P_{\sigma(x')}}$. Such a finite covering $C'$ of $C$ is easily constructed, for instance as a cyclic covering. 

Indeed, for any effective divisor $D$ in $C$ and any positive integer $a$ which divides $\deg D$, we may choose a line bundle $\mathcal{L}$ on $C$ and a non-zero  section $s$ of $\mathcal{L}^{\otimes a}$ such that $\mathrm{div}\, s = D$, and consider the reduced projective curve in the total space $\mathbb{V}(\mathcal{L}^\vee)$ of $\mathcal{L}$ defined as the $a$-th root of $s$. It is possibly singular over the support $\vert D \vert$ of $D$, and reducible.

 Any component $C'$ of its normalization is a connected smooth projective curve, and defines a finite covering  $\sigma: C' \ra C$ whose only ramification points lie over $\vert D \vert$. More precisely, the index of ramification of $\sigma$  at some point $x'$ lying over some point $x$ of order $n_x$ in $D$ is $e_x = a/ \mathrm{gcd}(a, n_x)$. 
 
 With a suitable choice of $a$ and $D$ (with support containing $\Delta$), this construction provides the desired covering.

\subsubsection{}\label{641bis}Consider the  fiber product of $X$ and $C'$ over $C$:
$$ X' := X \times_C C',$$
and denote by: 
$$\pi': X' \lra C' \quad \mbox{and} \quad \tau: X' \lra X$$
the two projections.

Moreover let us denote the set-theoretic inverse images of $H$ and $\Sigma$ by $\tau$ by:
$$H' := \tau^{-1} (H) \quad \mbox{and} \quad \Sigma':= \tau^{-1} (\Sigma),$$
which we shall also see as reduced\footnote{The divisor $H'$ coincides with the scheme-theoretic inverse image $\tau^\ast (H)$. However, if $\Sigma$ is not empty, then its scheme-theoretic inverse image by $\tau$ is not reduced.} subschemes of the smooth scheme $X'$.

Let:
$$\chi' : \widetilde{X}' \lra X'  \quad \mbox{(resp. $\chi : \widetilde{X} \lra X$)}$$
be the blow-up of  $\Sigma'$ (resp. $\Sigma$) seen as a reduced subscheme in  $X'$ (resp. in $X$), and let:
$$ Z' := \chi'^{-1}(\Sigma') \quad \mbox{(resp. $Z:= \chi^{-1}(\Sigma)$)}$$
be the exceptional divisor of $\chi'$ (resp. of $\chi$).  Its connected components are the inverse images:
$$Z'_{P'}:= \chi'^{-1}(P'), \quad P' \in \Sigma' \quad \mbox{(resp. $Z_{P}:= \chi^{-1}(P), \quad P \in \Sigma$)}.$$ 
Finally let  $\widetilde{H}'$ (resp. $\widetilde{H}$) be the strict transform of $H'$ (resp. $H$) in $\widetilde{X}'$ (resp. in $\widetilde{X}$).

The following morphism:
$$\nu' := \chi'_{| \widetilde{H}'} : \widetilde{H}' \lra  H'   \quad \mbox{(resp. $\nu:= \chi_{| \widetilde{H}} : \widetilde{H} \lra  H$)}$$
 can be identified with the blow-up of $\Sigma'$ (resp. $\Sigma$) in the complex scheme $H'$ (resp. in the smooth complex scheme $H$).
 
\subsubsection{}\label{642} These constructions are summarized by the following diagram where $Z'_\delta$ is the Cartier divisor in $\widetilde{X}'$ defined as:
 $$Z'_\delta := \sum_{P' \in \Sigma'} \delta_{\tau(P')}  Z'_{P'},$$
and where we identify the effective Cartier divisors $\widetilde{H}' + Z'_{\delta}$ and $ \widetilde{H} + Z$ with the codimension one subschemes they define:
$$\xymatrix{
& & Z' \ar[ld] \ar[ddd] & & & Z \ar[ld] \ar[ddd] \\
& \widetilde{H}' + Z'_{\delta} \ar[ld] \ar[ddd] & & & \widetilde{H} + Z \ar[ld] \ar[ddd] & \\
\widetilde{X}' \ar[ddd]_{\chi'} & & & \widetilde{X} \ar[ddd]^{\chi} & & \\
& & \Sigma' \ar[ld] \ar@{.>}[rrr] & & & \Sigma \ar[ld] \\
& H' \ar[ld] \ar[rrr]|(.655){\hole} & & & H \ar[ld] & \\
X' \ar[d]_{\pi'} \ar[rrr]^{\tau} & & & X \ar[d]^{\pi} & & \\
C' \ar[rrr]_{\sigma} & & & C & &}
$$
 
 Observe that, in this diagram, all rectangles with non-dotted arrows are cartesian.  
 
The divisor $\widetilde{H}'$ in $\widetilde{X}'$ is non-singular and intersects transversally the divisor $Z' = \sqcup_{P' \in \Sigma'} Z'_{P'}.$
Moreover the following equality of divisors in $\widetilde{X}'$ holds:
\begin{equation} \label{divisor H' in tilde X'}
\chi'^\ast (H') = \widetilde{H}' +  Z'_\delta.
\end{equation}

Similarly the divisor $\widetilde{H}$ in $\widetilde{X}$ is non-singular and intersects transversally the divisor $Z= \sqcup_{P \in \Sigma} Z_{P},$ and the following equality of divisors in $\widetilde{X}$ holds:
\begin{equation}\label{divisor H in tilde X}
\chi^*(H) = \widetilde{H} + Z.
\end{equation}

Moreover we denote by $E'_{P'}$ (resp. $E_P$) the intersection $Z'_{P'} \cap \widetilde{H}'$ (resp. $Z_P \cap \widetilde{H}$) for every point~$P'$ in $\Sigma'$ (resp. $P$ in $\Sigma$), by $H'_{\Delta'}$ and $\widetilde{H}'_{\Delta'}$ (resp. by $H_\Delta$ and $\widetilde{H}_\Delta$)  the inverse images of the reduced divisor $\Delta'$ in $C'$ (resp. of the reduced divisor $\Delta$ in $C$) by $\pi'_{| H'}$ and $\pi'_{| H'} \circ \nu'$ respectively (resp. by~$\pi_{| H}$ and $\pi_{| H} \circ \nu$), and by $\widetilde{H'_{\Delta'}}$ (resp. by $\widetilde{H_\Delta}$) the strict transform in $\widetilde{H}'$ (resp. in $\widetilde{H}$) of the divisor~$H'_{\Delta'}$ (resp. $H_\Delta$).

The divisor $\widetilde{H'_{\Delta'}}$ is non-singular and intersects transversally the non-singular divisor $E' := \sqcup_{P' \in \Sigma'} E'_{P'},$ and the following equality of divisors in $\widetilde{H}'$ holds:
\begin{equation} \label{sing fibers semihom '}
\widetilde{H}'_{\Delta'} = \widetilde{H'_{\Delta'}} + \sum_{P' \in \Sigma'} E'_{P'}.
\end{equation}

Similarly the divisor $\widetilde{H_{\Delta}}$ is non-singular and intersects transversally the non-singular divisor $E := \sqcup_{P \in \Sigma} E_{P},$ and the following equality of divisors in $\widetilde{H}$ holds:
\begin{equation}\label{sing fibers semihom}
\widetilde{H}_{\Delta} = \widetilde{H_{\Delta}} + \sum_{P \in \Sigma}  \delta_{P} E_{P}.
\end{equation}

In particular, the divisor $\widetilde{H}'_{\Delta'}$ in $\widetilde{H}'$ is a reduced divisor with strict normal crossings, and the divisor $\widetilde{H}_{\Delta}$ in $\widetilde{H}$ is a divisor with strict normal crossings, non-reduced when $\Sigma$ is not empty. This allows us to consider the vector bundles of relative logarithmic differentials $\omega^1_{\widetilde{H}'/C'}$ on $\widetilde{H}'$ and $\omega^1_{\widetilde{H}/C}$ on~$\widetilde{H}$.

\subsubsection{Computations of topological Euler characteristics}
\begin{proposition}\label{alpha_x on tilde(H)' semi-hom}
For every point $P'$ in $\Sigma'$ with image $P := \tau(P')$ in $\Sigma,$ the following equality of integers holds:
\begin{equation}\label{eq chi 2-codim strata semihom}
\chi_{\mathrm{top}}(E'_{P'} \cap \widetilde{H'_{\Delta'}} ) = (-1)^N / \delta_P  \;[(\delta_P-1)^N + (-1)^N (N \delta_P - 1)],
\end{equation}
where $\chi_{\mathrm{top}}$ denotes the topological Euler characteristic.
\end{proposition}

\begin{proof} When $N=1,$ the intersection $E'_{P'} \cap \widetilde{H'_{\Delta'}}$ is empty, and both sides of \eqref{eq chi 2-codim strata semihom} vanish. Therefore we may and shall assume $N$ to be at least $2$.

Consider $P'$ and $P$ as above. As, over some open neighborhood of $P'$, the complex scheme~$\widetilde{X}'$ is the blow-up at $P'$ of the smooth $(N+1)$-dimensional complex scheme $X'$, there is a canonical  isomorphism:
$$\psi : Z'_{P'} \lrasim \PP(T_{P'} X').$$

Under this isomorphism, the divisor $E'_{P'} = Z'_{P'} \cap \widetilde{H}'$ in $Z'_{P'}$ is mapped precisely onto the projective tangent cone $\PP(C_{P'} H')$ at $P'$ of the  hypersurface $H'$ in $X'$, of which $P'$ is an isolated singular point. By using Proposition \ref{Local H}, this projective tangent cone  is easily seen to be smooth of degree $\delta_P.$

Moreover, under the isomorphism $\psi$, the subscheme $E'_{P'} \cap \widetilde{H'_{\Delta'}}$ of $Z'_{P'}$  is mapped onto the projective tangent cone $\PP(C_{P'} H'_{\Delta'})$  at $P'$ of the subscheme $H'_{\Delta'} = H' \cap X'_{\Delta'}$ of $X'$, or equivalently onto the intersection of the projective tangent cone $\PP(C_{P'} H')$ and the projective tangent space $\PP(T_{P'} X'_{\Delta'})$, which is well-defined because $\pi'$ is smooth.

If we let: 
$$h := c_1(\cO_{T_{P'}X'_{\Delta'}} (1)) \quad \big(\,\mbox{in } \CH^1(\PP(T_{P'}X'_{\Delta'}))\big),$$
and:
$$I := \PP(C_{P'} H') \cap \PP(T_{P'} X'_{\Delta'}) = \psi(E'_{P'} \cap \widetilde{H'_{\Delta'}}),$$
which is a smooth hypersurface of degree $\delta_P$ in the projective space $\PP(T_{P'}X'_{\Delta'})$,  a standard computation of Chern classes yields the following equalities, where we have denoted by $[V]$ the class in the Grothendieck group~$K^0(I)$ of some vector bundle $V$ on $I$: 
\begin{align*}
\chi_{\mathrm{top}}(E'_{P'} \cap \widetilde{H'_{\Delta'}}) & = \int_{E'_{P'} \cap \widetilde{H'_{\Delta'}}} c_{N-2} (T_{E'_{P'} \cap \widetilde{H'_{\Delta'}}}) \\
& = \int_{I} c_{N-2} (T_{I}) \\
& = \int_{I} c_{N-2} \big([T_{\PP(T_{P'} X'_{\Delta'}) \mid I }] -[N_{I} \PP(T_{P'} X'_{\Delta'})]\big)\\
& = \int_{I} c_{N-2} \big( [ T_{P'} X'_{\Delta'} \otimes \cO_{T_{P'} X'_{\Delta'}} (1)_{\mid I}] - [\cO_{I}] - [\cO_{T_{P'}X'_{\Delta'}} (\delta_P)_{\mid I}] \big)\\
& =  \int_{I} (1+ h)^N (1 + \delta_P h)^{-1} \\
& =  \delta_P \int_{\PP(T_{P'} X'_{\Delta'})} h (1+ h)^N (1 + \delta_P h)^{-1}.
\end{align*}

Finally, using  equality \eqref{non square denominator semihom} applied to $n := N,$  $r := 2$ and $a := \delta_P$, we get:
\begin{equation}\label{chi of exceptional divisor '}
\chi_{\mathrm{top}}(E'_{P'}\cap \widetilde{H'_{\Delta'}}) =   \delta_P \, \bigg[\frac{(1+y)^{N}}{1+\delta_P y} \bigg]^{[N-2]}
=(-1)^N / \delta_P  \; [(\delta_P-1)^N + (-1)^N (N \delta_P - 1)]. 
\end{equation}
This concludes the proof of Proposition \ref{alpha_x on tilde(H)' semi-hom}. 
See also \cite[Proposition 5.2.1]{Mordant22} for a similar computation.
\end{proof}

\section{Characteristic classes of K\"ahler and logarithmic relative differentials on the hypersurfaces $H$, $\widetilde{H}$, and $\widetilde{H}'$}\label{CharRelDiff}

\subsection{Comparing $\omega^1_{\widetilde{H}/C}$ and $\omega^1_{\widetilde{H}'/C'}$}

The following proposition is arguably 
the point where the proof of Theorem \ref{GK semi hom in X} differs the most significantly from the proof of  \cite[Proposition~6.1.1]{Mordant22}.  This result, concerning the geometry of the morphisms: 
$$\pi_{\mid H} \circ \nu : \widetilde{H} \lra C$$
and: 
$$\pi'_{\mid H'} \circ \nu' : \widetilde{H}' \lra C', $$
will allow us to compute the stable Griffiths height $\hgt_{GK, stab}\big (\H^{N-1}(H_\eta / C_\eta)\big )$ without controlling the elementary exponents of the degeneration $\pi_{\mid H}: H \ra C.$

\begin{proposition}\label{both omega iso} 
With the notation introduced in \ref{641}, \ref{641bis}, and \ref{642}, the inverse image $(\tau_{| H'} \circ \nu')^{\ast}(\Sigma),$ seen as a subscheme of $\widetilde{H}'$, is precisely the exceptional divisor $E' = \sqcup_{P'\in \Sigma'} E'_{P'}$. In particular, it is a (Cartier) divisor in~$\widetilde{H}'$, and consequently 
there
exists a unique morphism $\upsilon : \widetilde{H}' \ra  \widetilde{H}$ such that the following diagram is commutative:
\begin{equation}\label{diag ups}
\xymatrix{
\widetilde{H}' \ar[r]^{\upsilon} \ar[d]_{\nu'} & \widetilde{H}\ar[d]^{\nu} \\
H' \ar[r]_{\tau_{| H'}} & H .
}
\end{equation}

If we denote by:
$$\phi : \tau_{| H'}^* \Omega^1_{H/C} \lrasim \Omega^1_{H'/C'}$$ 
the isomorphism of coherent sheaves on $H'$ induced by the cartesian diagram:
$$
\xymatrix{
H' \ar[r]^{\tau_{| H'}} \ar[d]_{\pi'_{| H'}} & H \ar[d]^{\pi_{| H}} \\
C' \ar[r]_{\sigma} & C,
}
$$
then the isomorphism of locally free sheaves on $\widetilde{H}' \setminus \nu'^{-1}(\Sigma') \simeq H' \setminus \Sigma'$ defined by the restriction of~$\phi$ extends to an isomorphism of locally free sheaves on $\widetilde{H}'$:
$$ \widetilde{\phi} : \upsilon^* \omega^1_{\widetilde{H}/C} \lrasim \omega^1_{\widetilde{H}'/C'}.$$
\end{proposition}

As indicated in \ref{123} above, we defer the proof of Proposition \ref{both omega iso} to Section \ref{Proof both omega}. 

\subsection{Characteristic classes of $\omega^{1}_{\widetilde{H}/C}$}

\subsubsection{}

The following proposition is a generalization of \cite[Lemma 5.1.1]{Mordant22}. 

\begin{proposition}\label{c omega on H}
The following equality holds in $\CH^N(H)$:
\begin{equation} \label{eq nu ast c1 cN-1 omega semihom}
\nu_* (c_1 c_{N-1})(\omega^{1 \vee}_{\widetilde{H}/C})
= (c_1 c_{N-1})([T_{\pi_{| H}}])
+ (-1)^{N-1} \sum_{P \in \Sigma} (1-N/\delta_P)\, \big((\delta_P-1)^N + (-1)^{N+1}\big) \, [P].
\end{equation}
\end{proposition}

In the right-hand side of \eqref{eq nu ast c1 cN-1 omega semihom}, by $[T_{\pi_{| H}}]$ we denote as usual  the relative tangent class~$[T_H] -\pi_{\mid H}^\ast [T_C]$ in~$K^0(H)$ of the morphism $\pi_{| H}:H \ra C$.

The remainder of this subsection is devoted to the proof of Proposition \ref{c omega on H}. 
The method we shall use for this proof is the same  as the one used in the derivation of \cite[Lemma 5.1.1]{Mordant22}, and we shall simply state without proof the intermediate results generalizing the ones in \cite[Sections 5.2 and 5.3]{Mordant22} which were used to establish \cite[Lemma 5.1.1]{Mordant22}.

\subsubsection{}

Let us first introduce some notation and recall several results of \cite[5.2]{Mordant22}, which  did not assume any condition about the singularities of the morphism $\pi_{| H} : H \ra C$, and therefore hold in the present setting.

As in \cite{Mordant22}, for every $P$ in $\Sigma$, we introduce the following characteristic class:
$$\eta_P := c_1(\cO_{\widetilde{H}} (E_P) ) \in \CH^1(\widetilde{H}).$$

As in \cite[Proposition 5.2.1]{Mordant22}, for every point $P$ in $\Sigma,$ the following equality holds in $\CH^N(H)$:
\begin{equation} \label{eq nu_* eta_p^N semihom}
\nu_\ast \eta_P^N = (-1)^{N-1} [P],
\end{equation}
and for every non-negative integer $r,$ the following equality holds in $\CH^r(E_P)$:
\begin{equation} \label{eq c_r tangent exceptional divisor}
c_r(T_{E_P}) = (-1)^r \binom{N}{r} \eta_{P | E_P}^r.
\end{equation}

For every $P$ in $\Sigma$ and for every non-negative integer $r,$ we may also directly apply equality \cite[(5.2.5)]{Mordant22} to the smooth hypersurface $Q_P := \widetilde{H_\Delta} \cap E_P$ of degree $m_P := \delta_P$ in the projective space $E_P$, and  get the following equality in $\CH^{r+1}(E_P)$:
\begin{equation}\label{eq c_r tangent delta-ic}
i_{\widetilde{H_\Delta} \cap E_P, E_P \ast} c_r(T_{\widetilde{H_\Delta} \cap E_P}) = (-1)^{r+1} \delta_P \bigg [\frac{(1 + y)^N}{1 + \delta_P y} \bigg ]^{[r]} \eta_{P | E_P}^{r+1},
\end{equation}
where $i_{\widetilde{H_\Delta} \cap E_P, E_P}$ denotes the inclusion of $\widetilde{H_\Delta} \cap E_P$ in $E_P$ and where, as above, $f(y)^{[r]}$ denotes the coefficient of $y^r$ in some formal series $f(y) \in \C[[y]]$.

Moreover, as in \cite[Corollary 5.2.4]{Mordant22}, for every positive integer $r,$ the following equality holds in $\CH^r(\widetilde{H})$:
\begin{equation}\label{eq T rel blowup semihom} 
c_r([T_{\pi_{| H} \circ \nu}]) = \nu^\ast c_r([T_{\pi_{| H}}]) + (-1)^r \Bigg [\binom{N}{r} - \binom{N}{r-1} \Bigg ] \sum_{P \in \Sigma} \eta_P^r.
\end{equation}

\subsubsection{}
Applying the comparison in \cite[Proposition 3.1.7, (3.1.12)]{Mordant22} of the Chern classes of the vector bundle $\omega^{1 \vee}_{\widetilde{H}/C}$ and the relative tangent class $[T_{\pi_{| H} \circ \nu}]$  in the case of a degeneration whose singular fibers are divisors with strict normal crossings with empty 3-codimensional strata\footnote{This is the case here because of the expression \eqref{sing fibers semihom} for the divisor $\widetilde{H}_\Delta$.}, we obtain the following equality in $\CH^r(\widetilde{H})$ for every non-negative integer $r$:
\begin{equation}\label{eq cr omega with strata}
 c_r(\omega^{1 \vee}_{\widetilde{H}/C}) = c_r([T_{\pi_{| H} \circ \nu}]) + \sum_{P \in \Sigma} (\delta_P-1)\, i_{E_P *} c_{r-1}(T_{E_P}) - \sum_{P \in \Sigma} \delta_P\,  i_{\widetilde{H_\Delta} \cap E_P *} c_{r-2}(T_{\widetilde{H_\Delta}}  \cap E_P).
\end{equation}
Observe that, in the right-hand side of \eqref{eq cr omega with strata}, Chern classes of negative degree arise when $r=0$ or $r=1$. These are zero by definition.

Combining equality \eqref{eq cr omega with strata} with equalities \eqref{eq c_r tangent exceptional divisor} applied to $r' := r-1$ and \eqref{eq c_r tangent delta-ic} applied to $r' := r-2$ yields the following equality in $\CH^r(\widetilde{H})$ generalizing \cite[Proposition 5.3.2]{Mordant22}, valid for every non-negative integer $r$:
\begin{equation}\label{eq cr omega T_g semihom} 
c_r(\omega^{1 \vee}_{\widetilde{H}/C}) = c_r([T_{\pi_{| H} \circ \nu}]) + \sum_{P \in \Sigma} \alpha(N,r,\delta_P)\,  \eta_P^r,
\end{equation}
where $\alpha(N,r,\delta)$ is the integer defined by:
$$\alpha(N,r,\delta) = (-1)^{r-1} \Bigg[(\delta - 1) \binom{N}{r-1} - \delta^2 \bigg[\frac{(1 + y)^N}{1 + \delta y}\bigg]^{[r-2]} \Bigg].$$

When $r$ is positive, combining equality \eqref{eq cr omega T_g semihom} with \eqref{eq T rel blowup semihom} yields the following equality in $\CH^r(\widetilde{H})$ generalizing \cite[Corollary 5.3.3]{Mordant22}:
\begin{equation}\label{eq cr omega nu ast semihom}
c_r(\omega^{1 \vee}_{\widetilde{H}/C}) = \nu^* c_r([T_{\pi_{| H}}]) + \sum_{P \in \Sigma} \beta(N,r,\delta_P) \, \eta_P^r,
\end{equation}
where $\beta(N,r,\delta)$ is the integer defined by:
$$\beta(N,r,\delta) = (-1)^{r} \Bigg[\binom{N}{r} - \delta \binom{N}{r-1} + \delta^2 \bigg[\frac{(1 + y)^N}{1 + \delta y}\bigg]^{[r-2]} \Bigg].$$

When $N$ is at least 2, multiplying equality \eqref{eq cr omega nu ast semihom} applied to $r := 1$ with the same equality applied to $r := N-1,$ and using \eqref{non square denominator semihom} applied to $n := N,$  $r := 3$ and $a := \delta_P$,
and the fact that the restriction to $E_P$ of the class $\nu^* [T_{\pi_{| H}}]$ is the class of a trivial vector bundle, we obtain the following equality in $\CH^N(H)$ generalizing \cite[Corollary 5.3.4]{Mordant22}: 
\begin{equation}\label{eq c1 cN-1 omega and nu ast semihom}
 (c_1 c_{N-1})(\omega^{1 \vee}_{\widetilde{H}/C}) = \nu^* (c_1 c_{N-1})([T_{\pi_{| H}}]) + \sum_{P \in \Sigma} (1 - N/\delta_P)\, \big((\delta_P-1)^N + (-1)^{N+1} \big)\,  \eta_P^N.
\end{equation}

When $N=1$, equality \eqref{eq c1 cN-1 omega and nu ast semihom} follows directly from equality \eqref{eq cr omega nu ast semihom} applied to $r:=1$. 

Pushing forward equality \eqref{eq c1 cN-1 omega and nu ast semihom} by $\nu$  and using equality \eqref{eq nu_* eta_p^N semihom} concludes the proof. 

\section{The geometry of the hypersurface $H$ in the  smooth pencil $X$  over $C$}\label{geom hyp smooth pencil}

\subsection{The top Chern class of $\Omega^1_{X/C | H} \otimes L_{| H}$}\label{subsec top Chern sing semihom}
Let us define a morphism $s$ of vector bundles on $H$ by the composition:
$$
s : L^\vee _{| H} \simeq \cN_H X^\vee \lra \Omega^1_{X \mid H} \lra \Omega^1_{X/C | H}.
$$
We will also see $s$ as a section of the vector bundle $\Omega^1_{X/C | H} \otimes L_{| H}$.

The following proposition is a generalization of the second statement of \cite[Lemma 6.1.3]{Mordant22}.

\begin{proposition}\label{Milnor semi hom}
The section $s$ is a regular section\footnote{in the sense of \cite[B.3.4]{Fulton98}.} of $\Omega^1_{X/C | H} \otimes L_{| H}$, and the 0-cycle in $H$ defined by its vanishing is given by:
$$[(s=0)] = \sum_{P \in \Sigma} (\delta_{P}-1)^N [P].$$
\end{proposition}

\begin{proof}
By definition, the subset of $H$ defined by the vanishing of $s$ is precisely the set of critical points of the morphism $\pi_{| H} : H \ra C,$ namely $\Sigma$. The regularity of the section $s$ follows from the fact that $\Sigma$ has codimension $N$ in the smooth $N$-dimensional manifold $H$, while the vector bundle~$\Omega^1_{X/C | H} \otimes L_{| H}$ is of rank $N$.

Now let us compute the multiplicity of the $0$-cycle $[(s = 0)]$ at some point $P$ in $\Sigma.$

As in Subsection \ref{local desc} and Proposition \ref{Local H} above, we denote by
$t$ a local coordinate of $C$ near~$\pi(P)$, and we complete its pullback $\pi^* t$  into a local coordinate system $(\pi^* t, x_1,\dots,x_N)$ of $X$ near $P$.  
We still denote by $F_{\delta_P}(T_1, \dots, T_N)$  an equation in $\C[T_1, \dots, T_N]$ of the projective tangent cone $\PP(C_P H_{\pi(P)})$, seen as a projective hypersurface in $\PP(T_P H)$ that we identify with $\PP^{N-1}_\C$ through the (differential at $P$ of the) coordinate system $(x_{1 | H},\dots,x_{N | H})$. It is a  homogeneous  polynomial of  degree $\delta_P$, with non-zero discriminant. 

According to Proposition \ref{Local H}, we may assume that the germ  at $P$ of the divisor $H$ in $X$ is 
 defined by the equation:
\begin{equation} \label{eq H in X F delta Milnor}
 \pi^* t = F_{\delta_P}(x_1,\dots,x_N) + F_{> \delta_P},
 \end{equation}
for some
 $F_{>\delta_P}$ in $\mathfrak{m}_{X,P}^{\delta_P +1}$.

The local trivialization of the line bundle $L:= \cO_X(H)$ on $X$ given by this equation and the local frame of the vector bundle $\Omega^1_{X/C}$ on $X$ defined by the relative differentials $([d x_1],\dots, [d x_N])$ induce by restriction and tensor product a local frame of the vector bundle $\Omega^1_{X/C | H} \otimes L_{| H}$ on $H$. In this local frame, the section $s$ is given by:
$$s = \big(\partial_{x_{1 | H}} (F_{\delta_P}(x_{1 | H},\dots,x_{N | H}) + F_{> \delta_P | H}), \dots, \partial_{x_{N | H}} (F_{\delta_P}(x_{1 | H},\dots,x_{N | H}) + F_{> \delta_P | H})\big).$$

Consequently, using that the equality \eqref{eq H in X F delta Milnor} holds on $H,$ we see that the multiplicity at $P$ of the $0$-cycle $(s = 0)$ is precisely the Milnor number of the singularity at $P$ of the morphism~$\pi_{| H} : H \ra C.$

The fact that this multiplicity is equal to $(\delta_P-1)^N$ is therefore a consequence of the computation of the Milnor number of an isolated semiquasihomogeneous singularity, in the particular semihomogeneous case where the weights are all equal to 1. This computation was done in \cite{Milnor-Orlik70} in the quasihomogeneous case where $F_{> \delta_P}$ vanishes; the reduction of the semiquasihomogeneous case to this case appears for instance in \cite[Theorem p. 194]{AGZV85}. \end{proof}

From Proposition \ref{Milnor semi hom}, we may derive the equality \eqref{sing hom hyp}  as in the proof of \cite[(6.1.10)]{Mordant22}, simply replacing the 0-cycle $i_* [\Sigma]$ in \emph{loc. cit.} by $\sum_{P \in \Sigma} (\delta_P-1)^N [P].$ This proof uses the fact that the cycle~$[(s = 0)]$ defined by the vanishing of the regular section $s$ is precisely the top Chern class of the vector bundle $\Omega^1_{X/C | H} \otimes L_{| H}$ (see for instance \cite[Example 3.2.16, (ii)]{Fulton98}), together with simple computations of Chern classes of vector bundles.

These computations also yield the following equality in $\CH_0(X)$, which shall be useful to us later:
\begin{equation}\label{sing loc as sum hom} 
\sum_{P \in \Sigma} (\delta_{P}-1)^N \,[P] = \sum_{k=0}^{N} c_1(L)^{N+1-k} c_k(\Omega^1_{X/C}).
\end{equation}

\subsection{Characteristic classes of the hypersurfaces $H$ and $H'$}\label{Char classes H H'}

With the notation introduced in \ref{641}, \ref{641bis}, and \ref{642}, let us denote by $\eta'$  the generic point of the curve $C'.$ Recall that the complex scheme~$\widetilde{H}'$ is smooth,  and that the divisor $\widetilde{H}'_{\Delta'}$ is reduced with strict normal crossings as shown by equation \eqref{sing fibers semihom '}, so that for every integer $n,$ the local monodromy at every point in $\Delta'$ of the variation of Hodge structures:
$$\H^n(H'_{\eta'} / C'_{\eta'}) \simeq \H^n(\widetilde{H}'_{\eta'} / C'_{\eta'})$$
is unipotent.

The following proposition is an analogue of \cite[Proposition 6.1.5]{Mordant22}.

\begin{proposition} 
\label{GK Lefschetz semihom}
If the line bundle $L':= \tau^\ast L$ on $X'$ is ample relatively to the morphism $\pi'$, then for every integer $n$ such that $n < N-1$, the following equality holds in $\CH_0(C')$:
\begin{equation}\label{eq GK Lefschetz small n semihom}
c_1\big(\GK_{C'}(\H^n(H'_{\eta'} / C'_{\eta'}))\big) = c_1\big(\GK_{C'}(\H^n(X'/C'))\big), 
\end{equation}
and for every integer $n$ such that $n > N-1$, the following equality holds in $\CH_0(C')_\Q$:
\begin{equation}\label{eq GK Lefschetz large n semihom}
c_1\big(\GK_{C'}(\H^n(H'_{\eta'} / C'_{\eta'}))\big) = c_1\big(\GK_{C'}(\H^{n+2}(X'/C'))\big). 
\end{equation}
\end{proposition}

\begin{proof} The proof uses the  arguments used in the proof of \cite[Proposition 6.1.5]{Mordant22}. Namely, for proving \eqref{eq GK Lefschetz small n semihom}, we apply the weak Lefschetz theorem to the smooth horizontal hypersurface~$H' \setminus H'_{\Delta'}$ in~$X' \setminus X'_{\Delta'}$ to obtain an isomorphism, for every $n < N-1,$ between the variations of Hodge structures~$\H^n\big((H' \setminus H'_{\Delta'}) / (C' \setminus \Delta')\big)$ and~$\H^n\big((X' \setminus X'_{\Delta'}) / (C' \setminus \Delta')\big)$ on $C' \setminus \Delta'.$ Then we use the unipotence of the local monodromy of both variations of Hodge structures to extend this isomorphism to an isomorphism of Deligne extensions over $C',$ and therefore of Griffiths line bundles.

Moreover equality \eqref{eq GK Lefschetz large n semihom} follows from equality \eqref{eq GK Lefschetz small n semihom} applied to the integer:
$$n' := 2(N-1)-n < N -1,$$
and from the consequence of Poincar\'e duality \cite[Proposition 2.5.2]{Mordant22}, applied to the degeneration $\widetilde{H}' / C'$ whose singular fibers form a reduced divisor with strict normal crossings and to the smooth degeneration $X'/C'.$ Indeed, applied to these data,  \cite[Proposition 2.5.2]{Mordant22} asserts  that for every $n > N-1,$ the following line bundles on $C'$: 
$$\GK_{C'}\big(\H^{2(N-1)-n}(H'_{\eta'} / C'_{\eta'})\big) \otimes \GK_{C'}\big(\H^n(H'_{\eta'} / C'_{\eta'} )\big)^\vee$$
and:
$$\GK_{C'}\big(\H^{2(N-1)-n}(X'/C')\big) \otimes \GK_{C'}\big(\H^{n+2}(X'/C')\big)^\vee$$
are 2-torsion.
\end{proof}

The following proposition is a generalization of \cite[(6.1.11)]{Mordant22}.
\begin{proposition} \label{class tangent hyp semihom} 
Denoting by $i$ the closed inclusion of $H$ in $X,$ the following equality holds in~$\CH_0(X)$:
\begin{multline}
i_\ast (c_1 c_{N-1})([T_{\pi_{\mid H}}]) = (c_1 c_N)(T_\pi) \\+ (-1)^N \big([(1 - c_1(L))^{-1} c_1(\Omega^1_{X/C})
c(\Omega^1_{X/C}) ]^{(N+1)} + \sum_{P \in \Sigma} (\delta_P-1)^N [P] - c_1(L) c_N(\Omega^1_{X/C}) \big). \label{eq class tangent hyp}
\end{multline}
\end{proposition}

\begin{proof}
The proof relies on the same computations of Chern classes on $H$ and $X$ as in the proof of \cite[(6.1.11)]{Mordant22}, simply replacing the 0-cycle $i_* [\Sigma]$ in \emph{loc. cit.} by $\sum_{P \in \Sigma} (\delta_P-1)^N [P]$, using the expression of this sum given by equality \eqref{sing loc as sum hom}, and using that the morphism of smooth schemes~$\pi_{| H} : H \ra C$ is smooth on a dense open subset of $H$ so that the following equality holds in the Grothendieck group $K^0(H)$:
\begin{equation*}
[T_{\pi_{\mid H}}] = [\Omega^1_{H/C}]^\vee.  \qedhere
\end{equation*}
\end{proof}

\subsection{Completing the proof of Theorem \ref{GK semi hom in X}}\label{complete}
As equality \eqref{sing hom hyp} was established in Subsection \ref{subsec top Chern sing semihom}, to complete the proof of Theorem \ref{GK semi hom in X}, it only remains to show equality \eqref{eq: GK semi hom in X} under the hypothesis that the line bundle $L$ on $X$ is ample relatively to the morphism $\pi.$

\subsubsection{The alternating sum $\sum_{n = 0}^{2(N-1)} (-1)^{n-1} \hgt_{GK}\big(\H^n(H'_{\eta'} / C'_{\eta'}) \big)$
} We shall use the notation introduced in \ref{641}, \ref{641bis}, and \ref{642} and denote by $\eta'$  the generic point of the curve $C'$. As already observed at the beginning of \ref{Char classes H H'},   for every integer $n,$ the local monodromy at every point in $\Delta'$ of the variation of Hodge structures:
$$\H^n(H'_{\eta'} / C'_{\eta'}) \simeq \H^n(\widetilde{H}'_{\eta'} / C'_{\eta'})$$
is unipotent.

Recall also that by hypothesis the restriction $\pi_{| \Sigma} : \Sigma \ra C$ is injective, so that the restriction~$\pi'_{| \Sigma'} : \Sigma' \ra C'$ also is injective.

Consequently, applying the special case of Theorem \ref{GK in terms of classes in D non-reduced} stated in  equality \eqref{eq GK in terms of classes in D reduced semihom} to the degeneration~$\widetilde{H}'/C',$ which over $\eta'$ coincides with $H'/C'$ and whose divisor of singular fibers $\widetilde{H}'_{\Delta'}$ is a reduced divisor with strict normal crossings as shown by equality \eqref{sing fibers semihom '}, the following equality holds:
\begin{equation}\label{first eq computation GK stab semihom}
\sum_{n = 0}^{2(N-1)} (-1)^{n-1} \hgt_{GK}\big(\H^n(H'_{\eta'} / C'_{\eta'} ) \big)
= 1/12 \int_{\widetilde{H}'} (c_1 c_{N-1})(\omega^{1 \vee}_{\widetilde{H}'/C'}) + 1/12 \, \chi_{\mathrm{top}}(\widetilde{H'_{\Delta'}} \cap E').
\end{equation}

Using Proposition \ref{alpha_x on tilde(H)' semi-hom}, for every point $P'$ in $\Sigma'$ the Euler characteristic of the open and closed subset\footnote{Actually, a connected component when $N \geq 3$.}  $\widetilde{H'_{\Delta'}} \cap E'_{P'}$ of $\widetilde{H'_{\Delta'}} \cap E'$ is given by:
$$\chi_{\mathrm{top}}(\widetilde{H'_{\Delta'}} \cap E'_{P'}) =  (-1)^N / \delta_{\tau(P')}  \;  [(\delta_{\tau(P')}-1)^N + (-1)^N (N \delta_{\tau(P')} - 1)].$$

Since the morphism of smooth curves $\sigma : C' \ra C$ is ramified at every point $x'$ in~$\Delta'$ with ramification index $\delta_{P_{\sigma(x')}}$, where $P_{\sigma(x')}$ denotes the unique point in $\Sigma \cap \pi^{-1}(\sigma(x')),$ each point $x$ in~$\Delta$ has precisely~$\deg(\sigma)/\delta_{P_x}$ pre-images in $\Delta',$ and therefore each point $P$ in $\Sigma$ has precisely~$\deg(\sigma)/\delta_P$ pre-images in $\Sigma'$. Consequently the Euler characteristic of $\widetilde{H'_{\Delta'}} \cap E'$ can be written as follows:
\begin{align}
\chi_{\mathrm{top}}(\widetilde{H'_{\Delta'}} \cap E') &= \sum_{P \in \Sigma} \deg(\sigma)/\delta_P  \;  (-1)^N / \delta_P \;  [(\delta_P-1)^N + (-1)^N (N \delta_P - 1)] \nonumber \\
\label{eq sum alpha stab semihom} & = (-1)^N \deg(\sigma) \sum_{P \in \Sigma} 1 / \delta_P^2 \; [(\delta_P-1)^N + (-1)^N (N \delta_P - 1)].
\end{align}

Substituting \eqref{eq sum alpha stab semihom} into \eqref{first eq computation GK stab semihom} yields the following equality:
\begin{multline}\label{second eq computation GK stab semihom}
\sum_{n = 0}^{2(N-1)} (-1)^{n-1} \hgt_{GK}\big(\H^n(H'_{\eta'} / C'_{\eta'} ) \big)
= 1/12 \; \int_{\widetilde{H}'} (c_1 c_{N-1})(\omega^{1 \vee}_{\widetilde{H}'/C'}) \\ + (-1)^N \deg(\sigma)/12 \; \sum_{P \in \Sigma} 1 / \delta_P^2  \; [(\delta_P-1)^N + (-1)^N (N \delta_P - 1)].
\end{multline}

According to Proposition \ref{both omega iso}, the vector bundle $\omega^{1 \vee}_{\widetilde{H}'/C'}$ is isomorphic to $\upsilon^* \omega^{1 \vee}_{\widetilde{H}/C}$. Moreover the restriction of the morphism $\upsilon : \widetilde{H}' \ra \widetilde{H}$ to the dense open subset $\widetilde{H}' \setminus E' \simeq H' \setminus \Sigma'$ coincides with the morphism $$\tau_{| H' \setminus \Sigma'} : H' \setminus \Sigma' \lra H \setminus \Sigma$$ which is, by base change, a finite morphism of degree $\deg(\sigma).$ Therefore the morphism $\upsilon$ is generically finite, of degree $\deg(\sigma).$ Consequently, using the projection formula applied to $\upsilon$, we may rewrite equality \eqref{second eq computation GK stab semihom} as follows:
\begin{multline}\label{third eq computation GK stab semihom}
\sum_{n = 0}^{2(N-1)} (-1)^{n-1} \hgt_{GK}\big(\H^n(H'_{\eta'} / C'_{\eta'}) \big )
= \deg(\sigma)/12 \;  \int_{\widetilde{H}} (c_1 c_{N-1})(\omega^{1 \vee}_{\widetilde{H}/C})  \\ + (-1)^N \deg(\sigma)/12 \;  \sum_{P \in \Sigma} 1 / \delta_P^2 \; [(\delta_P-1)^N + (-1)^N (N \delta_P - 1)].
\end{multline}

Using Proposition \ref{c omega on H},   we may rewrite equality \eqref{third eq computation GK stab semihom} as follows: 
\begin{equation}\label{fourth eq computation GK stab semihom}
\sum_{n = 0}^{2(N-1)} (-1)^{n-1} \hgt_{GK}\big (\H^n(H'_{\eta'} / C'_{\eta'} ) \big)
= \deg(\sigma)/12 \; \int_{H} (c_1 c_{N-1})([T_{\pi_{| H}}]) + \deg(\sigma) \sum_{P \in \Sigma} u_{N,\delta_P},
\end{equation}
where $u_{N, \delta}$ is the rational number given by:
\begin{align*}
u_{N,\delta} & : = (-1)^{N-1}/12 \;  (1 - N/\delta) [(\delta-1)^N + (-1)^{N+1}] +  (-1)^N/(12 \delta^2) \; [(\delta-1)^N + (-1)^N (N \delta - 1)] \\
&\;  = (-1)^N (\delta-1) / (12 \delta^2) \;  [(N \delta - \delta^2 + 1) (\delta-1)^{N-1} + (-1)^N (\delta+1)].
\end{align*}

Using Proposition \ref{class tangent hyp semihom}, we may rewrite equality \eqref{fourth eq computation GK stab semihom} as follows:
\begin{multline}\label{fifth eq computation GK stab semihom}
\sum_{n = 0}^{2(N-1)} (-1)^{n-1} \hgt_{GK}\big (\H^n(H'_{\eta'} / C'_{\eta'}) \big )
= \deg(\sigma)/12 \; \int_X (c_1 c_N)(T_\pi) \\ + (-1)^N \deg(\sigma)/12 \; \bigg(\int_X (1 - c_1(L))^{-1} c_1(\Omega^1_{X/C}) c(\Omega^1_{X/C}) - \int_X c_1(L) c_N(\Omega^1_{X/C}) \bigg) \\
+ \deg(\sigma) \sum_{P \in \Sigma} v_{N,\delta_P},
\end{multline}
where $v_{N,\delta}$ is the rational number given by:
\begin{align}
v_{N,\delta} &:= u_{N,\delta} + (-1)^N (\delta-1)^N/12 \nonumber \\
\label{eq v w semihom} &= (-1)^N w_{N,\delta}.
\end{align}

\subsubsection{} Using again Theorem \ref{GK in terms of classes in D non-reduced}, this time applied to the smooth morphism $\pi' : X' \ra C',$ yields the following equality:
\begin{equation}\label{eq GK smooth semihom}
 \sum_{n = 0}^{2N} (-1)^{n-1} \hgt_{GK}\big(\H^n(X'/C')\big) = 1/12 \;  \int_{X'} (c_1 c_{N})(T_{\pi'}).
 \end{equation}

Moreover, using that by base change the line bundle $L'$ on $X'$ is ample relatively to the morphism~$\pi',$ since $L$ is ample relatively to $\pi$, Proposition \ref{GK Lefschetz semihom} implies the following equality:
\begin{multline}\label{comp alt sum GK stab semihom}
 \sum_{n=0}^{2(N-1)} (-1)^{n-1} \hgt_{GK}\big(\H^n(H'_{\eta'} / C'_{\eta'})\big) - \sum_{n=0}^{2N} (-1)^{n-1} \hgt_{GK}\big(\H^n(X'/C')\big) \\ = (-1)^N \Big[\hgt_{GK}\big(\H^{N-1}(H'_{\eta'} / C'_{\eta'})\big) - \hgt_{GK}\big(\H^{N-1}(X'/C')\big)
+ \hgt_{GK}\big(\H^N(X'/C')\big) \\ 
 - \hgt_{GK}\big(\H^{N+1}(X'/C')\big)\Big].
\end{multline}

Combining equalities \eqref{fifth eq computation GK stab semihom}, \eqref{eq v w semihom}, \eqref{eq GK smooth semihom}, and \eqref{comp alt sum GK stab semihom} yields the following equalities:
\begin{align}
\hgt_{GK}\big(\H^{N-1}(H'_{\eta'} / C'_{\eta'})\big)
& = \hgt_{GK}\big(\H^{N-1}(X'/C')\big) - \hgt_{GK}\big(\H^N(X'/C')\big) + \hgt_{GK}\big(\H^{N+1}(X'/C')\big)
 \nonumber \\
 & \quad - (-1)^N/12 \; \int_{X'} (c_1 c_{N})(T_{\pi'})
 + (-1)^N \deg(\sigma)/12 \; \int_X (c_1 c_N)(T_\pi) 
\nonumber \\ 
& \quad + \deg(\sigma)/12 \; \bigg(\int_X (1 - c_1(L))^{-1} c_1(\Omega^1_{X/C}) c(\Omega^1_{X/C}) - \int_X c_1(L) c_N(\Omega^1_{X/C})\bigg)
\nonumber \\ 
& \quad  + \deg(\sigma) \sum_{P \in \Sigma} w_{N,\delta_P}
\nonumber \\ 
 &= \deg(\sigma) \Big[ \hgt_{GK}\big(\H^{N-1}(X/C)\big) -  \hgt_{GK}\big(\H^N(X/C)\big) 
 + \hgt_{GK}\big(\H^{N+1}(X/C)\big)\Big] \nonumber \\
 & \quad  
+ \deg(\sigma)/12 \; \bigg(\int_X (1 - c_1(L))^{-1} c_1(\Omega^1_{X/C}) c(\Omega^1_{X/C})  - \int_X c_1(L) c_N(\Omega^1_{X/C}) \bigg)
\nonumber \\
 & \quad  
+ \deg(\sigma) \sum_{P \in \Sigma} w_{N,\delta_P}, \label{sixth eq computation GK stab semihom}
\end{align}
where in \eqref{sixth eq computation GK stab semihom} we  used that $\pi' : X' \ra C'$ is the base change of $\pi : X \ra C$ by $\sigma: C' \ra C.$

Since the local monodromy of the variation of Hodge structures $\H^{N-1}(H'_{\eta'} / C'_{\eta'})$ is unipotent, the stable Griffiths height $\hgt_{GK,stab}\big(\H^{N-1}(H_\eta/C_\eta)\big)$ is by definition given by:
$$\hgt_{GK,stab}\big(\H^{N-1}(H_\eta/C_\eta)\big) = 1/\deg(\sigma)\;  \hgt_{GK}\big(\H^{N-1}(H'_{\eta'} / C'_{\eta'})\big).$$

Combining this relation and equality \eqref{sixth eq computation GK stab semihom} concludes the proof of equality \eqref{eq: GK semi hom in X} and of Theorem~\ref{GK semi hom in X}.

\section{Proof of Proposition \ref{both omega iso}}\label{Proof both omega}

In this section, we use the notation introduced in Proposition \ref{both omega iso}, and we establish it.  
 
The existence of $\upsilon$ such that the diagram \eqref{diag ups} is commutative will follow from the fact that $(\tau_{| H'} \circ \nu')^{\ast}(\Sigma)$ is a Cartier divisor and from the universal property of blow-ups, and its uniqueness is clear. 

The assertions that  $(\tau_{| H'} \circ \nu')^{\ast}(\Sigma)$ and $E'$ coincide and that~$\phi$ extends to an isomorphism~$\widetilde{\phi}$ between ~$\upsilon^* \omega^1_{\widetilde{H}/C}$ and~$\omega^1_{\widetilde{H}'/C'}$ are local on $\widetilde{H}'$, and clearly satisfied outside of the closed sub\-scheme~$E':=\nu'^{\ast}(\Sigma').$ 
 
Consequently we may choose a point $P'$ in $\Sigma'$ and work locally near $E'_{P'}$ to establish these assertions.

We shall denote by $P$ the image of $P'$ by $\tau$:
$$P := \tau(P') \in \Sigma,$$ 
and, for simplicity's sake, we shall write $\delta$ for $\delta_P.$

Let us briefly describe the successive steps of the proof of Proposition \ref{both omega iso}. 

In Subsection \ref{coord on X and X'}, we introduce some compatible local systems of coordinates $(\pi^\ast t, x_1, \dots, x_N)$ and  $(\pi'^\ast t', x'_1, \dots, x'_N)$ on $X$ and $X'$ near $P$ and $P'$, respectively, and we describe (analytic) open coverings $(V_0, \dots, V_N)$  and $(V'_0, \dots, V'_N)$, attached to these systems, of neighborhoods of $Z_P$ and $Z'_{P'}$ in the blow-ups $\widetilde{X}$ and~$\widetilde{X}'$. 

These coordinates are used in Subsection \ref{subsec desc omega H tilde C V_1} to give local descriptions of the vector bundle 
$\omega^1_{\widetilde{H}/C}$ near a point $Q$ of $E_P\cap V_i$, $i \geq 1$.  We have to distinguish two cases, according as $Q$ belongs to~$\widetilde{H_\Delta}$  or not. 

In Subsection \ref{7.3}, we prove that the subscheme $E'_{P'}$ of $\widetilde{X}'$ is contained in the union of $V'_1, \dots, V'_N$. Consequently, to complete the proof of Proposition \ref{both omega iso}, it will be enough to consider the situation where the point $Q'$ in $E'_{P'}$, on some neighborhood of which we want to establish the validity of this proposition, belongs to one of the open subsets~$V'_1, \dots, V'_N$.  

Subsection~\ref{7.4} is devoted to the construction of $\upsilon$ and to the description of the pullback $\upsilon^* \omega^1_{\widetilde{H}/C}$ over $\widetilde{H}' \cap V'_1$. This relies on the local descriptions of $\omega^1_{\widetilde{H}/C}$ established in~Subsection \ref{subsec desc omega H tilde C V_1}.

Finally, in Subsection~\ref{7.5}, we provide a description of  the vector bundle $\omega^1_{\widetilde{H}'/C'}$ locally near $Q'$, and by comparing it with the one  of $\upsilon^* \omega^1_{\widetilde{H}/C}$ in Subsection~\ref{7.4}, we conclude the proof of Proposition~\ref{both omega iso}.

\subsection{Coordinate systems on $X$ and $X'$} \label{coord on X and X'}

\subsubsection{Notation}
As in Subsection \ref{local desc} and Proposition \ref{Local H} above, we denote by
$t$ a local coordinate of $C$ near $\pi(P)$, and we complete its pullback $\pi^* t$  into a local coordinate system $(\pi^* t, x_1,\dots,x_N)$ of $X$ near $P$.  

We still denote by $F_\delta(T_1, \dots, T_N)$  an equation in $\C[T_1, \dots, T_N]$ of the projective tangent cone $\PP(C_P H_{\pi(P)})$, seen as a projective hypersurface in $\PP(T_P H)$ that we identify with $\PP^{N-1}_\C$ through the (differential at $P$ of the) coordinate system $(x_{1 | H},\dots,x_{N | H})$. It is a  homogeneous  polynomial of  degree $\delta$, with non-zero discriminant. According to Proposition \ref{Local H}, we may assume that the germ  at $P$ of the divisor $H$ in $X$ is 
 defined by the equation:
\begin{equation} \label{eq H in X F delta}
 \pi^* t = F_\delta(x_1,\dots,x_N) + F_{> \delta},
 \end{equation}
 for some
 $F_{>\delta}$ in $\mathfrak{m}_{X,P}^{\delta_P +1}$.
 
Let $t'$ be a local coordinate  of $C'$ near $\pi'(P')$.  As the finite covering $\sigma : C' \ra C$ is ramified with index $\delta$ at $\pi'(P')$, there exists a 
unit $u$ in $\cO_{C', \pi'(P')}$ such that the following relation holds:
\begin{equation*}
\sigma^* t = u \, t'^\delta.
\end{equation*}
 Consequently the smooth complex scheme $X'$  admits  $(\pi'^* t', x'_1,\dots,x'_N)$ as local coordinate system near $P',$ where for every integer $i$ such that
$1 \leq i \leq N$:
$$ x'_i := \tau^* x_i.$$
Moreover the divisor $H'$ in $X'$ is defined by the following equation: 
\begin{equation}\label{eq H' in X' F delta}
 \pi'^*\!u \, \pi'^* t'^\delta = F_\delta(x'_1,\dots,x'_N) + \tau^* F_{> \delta}.
 \end{equation}

\subsubsection{Description of the blow-ups $\widetilde{X}'$ and  $\widetilde{X}$} As the $X'$-scheme $\widetilde{X}'$ (resp. the $X$-scheme $\widetilde{X}$) is locally the blow-up of $P'$ (resp. $P$) in $X'$ (resp. $X$), it can locally\footnote{Namely on some open neighborhood of $Z'_{P'} := \chi'^{-1} (P')$ (resp. of $Z_{P} := \chi^{-1} (P)$).} be identified with the closed subscheme of $X' \times \PP^N_\C$ (resp. $X \times \PP^N_\C$) defined by the following equations, where $(v'_0 : v'_1 : \dots : v'_N)$ (resp. $(v_0 : v_1 : \dots : v_N)$)  denote 
homogeneous coordinates on $\PP^N_\C$: 
$$v'_0\,  \chi'^* x'_i = \chi'^* \pi'^* t' \, v'_i  \quad \mbox{(resp. $v_0 \, \chi^* x_i = \chi^* \pi^* t \,  v_i$)}$$
for every integer $i \in \{1, \dots, N\}$, and: 
$$v'_i \, \chi'^* x'_j = \chi'^* x'_i \, v'_j \quad \mbox{(resp. $v_i \, \chi^* x_j = \chi^* x_i \, v_j$)}$$
for every pair of integers $(i,j)\in  \{1, \dots, N\}^2$ such that  $i \neq j$.

In particular, an open neighborhood of $Z'_{P'} := \chi'^{-1} (P')$ (resp. of $Z_{P} := \chi^{-1} (P)$) in the complex scheme $\widetilde{X}'$ (resp.  $\widetilde{X}$) admits a 
 covering by the open subsets $V'_0,\dots,V'_N$ (resp. $V_0,\dots,V_N$) defined 
 by the non-vanishing of $v'_0,\dots,v'_N$ (resp. of $v_0,\dots,v_N$).

\subsection{Description of $\omega^1_{\widetilde{H}/C \mid \widetilde{H} \cap V_1}$} \label{subsec desc omega H tilde C V_1}

For future use, let us state and prove two lemmas. The first one of them, Lemma \ref{lem om on V_1}, will describe the restriction of the vector bundle $\omega^1_{\widetilde{H}/C}$ to the open subset\footnote{The same reasoning would also allow us to describe the restriction of this vector bundle to $\widetilde{H} \cap V_i$ for any integer~$i$ in $\{2,\dots, N\}$.}~$\widetilde{H} \cap V_1$ of $\widetilde{H}$.

Let us first introduce some notation. The open subset $V_1$ admits as coordinate system\footnote{In other words, the morphism $(\widetilde{\pi^* t}_1, \chi^* x_1, \widetilde{x}_{2,1},\dots,\widetilde{x}_{N,1}): V_1 \ra \A_\C^{N+1}$ is \'etale.}:
$$(\widetilde{\pi^* t}_1, \chi^* x_1, \widetilde{x}_{2,1},\dots,\widetilde{x}_{N,1})$$ where $\widetilde{\pi^* t}_1$ is defined by:
$$\widetilde{\pi^* t}_1 := v_0/v_1$$
and satisfies the following equality:
\begin{equation}\label{t in V_1}
 \chi^* \pi^* t = \widetilde{\pi^* t}_1 \, \chi^* x_1,
 \end{equation}
and where for every integer $i \in \{2, \dots, N\}$, $\widetilde{x}_{i,1}$ is defined by: 
$$\widetilde{x}_{i,1}:= v_i/v_1$$
and satisfies the following equality:
\begin{equation}\label{x_i in V_1}
 \chi^* x_i = \widetilde{x}_{i,1} \,  \chi^* x_1.
 \end{equation}

 Observe moreover that the exceptional divisor $Z_P \cap V_1$ in $V_1$ is defined by the vanishing of $\chi^* x_1$.

\begin{lemma}\label{lem om on V_1}
The following expression:
$$h_1 := \chi^* x_1^{-\delta-1} \, \, \chi^\ast F_{> \delta}$$
defines a (germ of) regular function on $V_1$ along $Z_{P} \cap V_1$, and, letting $w$ be the (germ of) regular function defined as:
$$ w := F_\delta(1, \widetilde{x}_{2,1},\dots,\widetilde{x}_{N,1}) + \chi^* x_1 \,  h_1,$$
the divisor $\widetilde{H_\Delta} \cap V_1$ in $\widetilde{H} \cap V_1$ is locally defined by the following equation:
\begin{equation}\label{H Delta tilde in V_1}
w_{\mid \widetilde{H} \cap V_1} = 0.
\end{equation}

Finally, if $Q$ is any complex point of $E_P \cap V_1$, the vector bundle $\omega^1_{\widetilde{H}/C \mid \widetilde{H} \cap V_1}$ admits the following description near $Q$.
\begin{enumerate}
\item \label{case H Del tilde for omega} If $Q$ is in $\widetilde{H_\Delta}$ (or equivalently if the function $w$ vanishes at $Q$), then there exists an integer $i(Q)$ in $\{2,\dots,N\}$ such that the partial derivative
$$\partial_{\widetilde{x}_{i(Q),1}} F_\delta(1, \widetilde{x}_{2,1},\dots,\widetilde{x}_{N,1})$$
does not vanish at $Q$. Moreover, for any such integer $i(Q)$, the vector bundle $\omega^1_{\widetilde{H}/C \mid \widetilde{H} \cap V_1}$ admits the following local frame on some (analytic) open neighborhood of $Q$ in $\widetilde{H} \cap V_1$:
\begin{equation}\label{frame omega H tilde C}
\big(\chi_{\mid \widetilde{H} \cap V_1}^* [dx_1/x_1], \big[d \widetilde{x}_{2,1\mid \widetilde{H} \cap V_1}\big],\dots, \big[d \widetilde{x}_{i(Q)-1,1\mid \widetilde{H} \cap V_1}\big], \big[d \widetilde{x}_{i(Q)+1,1\mid \widetilde{H} \cap V_1}\big],\dots,\big[d \widetilde{x}_{N,1\mid \widetilde{H} \cap V_1}\big]\big).
\end{equation}

\item \label{case not H Del tilde for omega} If $Q$ is not in $\widetilde{H_\Delta}$ (or equivalently if the function $w$ does not vanish at $Q$), then the vector bundle $\omega^1_{\widetilde{H}/C \mid \widetilde{H} \cap V_1}$ admits the following local frame on some open neighborhood of $Q$ in~$\widetilde{H} \cap V_1$:
\begin{equation}\label{local frame omega Q not H Del tilde}
\big(\big[d \widetilde{x}_{2,1 \mid \widetilde{H} \cap V_1}\big],\dots, \big[d \widetilde{x}_{N,1 \mid \widetilde{H} \cap V_1}\big]\big).
\end{equation}
\end{enumerate}
\end{lemma}

The remainder of this subsection is devoted to the proof of Lemma \ref{lem om on V_1}.

Observe first that the regularity of the (germ of) function $h_1$ follows from the fact that $F_{> \delta}$ belongs to the ideal~$\mathfrak{m}_{X,P}^{\delta +1}$ of~$\cO_{X,P}$, and therefore the pullback $\chi^\ast F_{> \delta}$ vanishes at order at least $\delta+ 1$ on the exceptional divisor~$Z_{P}$.

\subsubsection{} Let us now describe the closed subschemes $\widetilde{H} \cap V_1$ and $\widetilde{H_\Delta} \cap V_1$ in $V_1$ to establish that \eqref{H Delta tilde in V_1} is indeed an equation for the divisor $\widetilde{H_\Delta} \cap V_1$ in $\widetilde{H} \cap V_1$.

Using firstly the equation \eqref{eq H in X F delta} of the divisor $H$ in $X$ and equalities \eqref{t in V_1} and \eqref{x_i in V_1}, and then the homogeneity of $F_\delta$ and the definition of $h_1$, the divisor $\chi^* H \cap V_1$ in $V_1$ is defined by the following equation:
\begin{align*}
\widetilde{\pi^* t}_1  \,  \chi^* x_1
& = F_\delta(\chi^* x_1,\widetilde{x}_{2,1}  \chi^* x_1, \dots,\widetilde{x}_{N,1}  \chi^* x_1) + \chi^\ast F_{> \delta} \\
& = \chi^* x_1^{\delta} \, \big[F_\delta(1, \widetilde{x}_{2,1},\dots,\widetilde{x}_{N,1}) + \chi^* x_1   \, h_1\big].
\end{align*}

Consequently the divisor $\widetilde{H} \cap V_1 = (\chi^* H - Z) \cap V_1$ in $V_1$ is defined by the following equation:
\begin{equation}\label{H tilde in V_1}
 \widetilde{\pi^* t}_1 = \chi^* x_1^{\delta-1} \,  w,
 \end{equation}
and admits  as coordinate system the restriction of $(\chi^* x_1, \widetilde{x}_{2,1},\dots,\widetilde{x}_{N,1})$ to $\widetilde{H} \cap V_1$.

Moreover, using equation \eqref{t in V_1}, the divisor $\widetilde{H}_\Delta \cap V_1$ in $\widetilde{H} \cap V_1$ is defined by the following equation:
$$\big(\widetilde{\pi^* t}_1  \,  \chi^* x_1\big)_{\mid  \widetilde{H} \cap V_1} = 0,$$
or equivalently, using that equality \eqref{H tilde in V_1} holds on $\widetilde{H} \cap V_1$:
\begin{equation}\label{H tilde Delta in V_1} 
 \big(  \chi^* x_1^{\delta} \,  w \big)_{\mid  \widetilde{H} \cap V_1}= 0.
\end{equation}

Consequently the divisor $\widetilde{H_\Delta} \cap V_1 = (\widetilde{H}_\Delta - \delta E_P) \cap V_1$ in $\widetilde{H} \cap V_1$ is defined by the equation \eqref{H Delta tilde in V_1}.

Let us also mention the following equalities of meromorphic differential forms on $\widetilde{H} \cap V_1$:
\begin{align}
\chi_{\mid \widetilde{H} \cap V_1}^* \pi^* (dt/t)
& = \big(d \widetilde{\pi^* t}_1 / \widetilde{\pi^* t}_1\big)_{\mid \widetilde{H} \cap V_1} + \chi_{\mid \widetilde{H} \cap V_1}^* (d x_1/x_1) \label{first eq comput dt sur t not H del tilde} \\
& = \delta \, \chi_{\mid \widetilde{H} \cap V_1}^* (d x_1/x_1) + (d w/w)_{\mid \widetilde{H} \cap V_1}, \label{second eq comput dt sur t not H del tilde}
\end{align}
where in \eqref{first eq comput dt sur t not H del tilde} we have used equality \eqref{t in V_1} and in \eqref{second eq comput dt sur t not H del tilde} we have used that equality \eqref{H tilde in V_1} holds on~$\widetilde{H} \cap V_1$.

\subsubsection{} Now let us consider a point $Q$ in $E_P \cap V_1$.

The goal of this paragraph is to establish the assertion \eqref{case H Del tilde for omega} of Lemma \ref{lem om on V_1}, namely to assume that~$Q$ is in $\widetilde{H_\Delta}$ and to prove that there exists an integer $i(Q)$ such that the partial derivative
$$\partial_{\widetilde{x}_{i(Q),1}} F_\delta(1, \widetilde{x}_{2,1},\dots,\widetilde{x}_{N,1})$$
does not vanish at $Q$, and that, for any such integer, a local frame of the vector bundle $\omega^1_{\widetilde{H}/C \mid \widetilde{H} \cap V_1}$ near $Q$ is given by \eqref{frame omega H tilde C}.

By hypothesis, the function $w$ vanishes at $Q$. Moreover, since $Q$ is in $Z_P$, the function $\chi^* x_1$ vanishes at $Q$. Hence by definition of $w$, the following function:
$$F_\delta(1, \widetilde{x}_{2,1},\dots,\widetilde{x}_{N,1})$$
vanishes at $Q$.

Since the homogeneous polynomial $F_\delta$ has non-zero discriminant, there exists an integer $i(Q)$ in $\{2,\dots,N\}$ such that the partial derivative
$$\partial_{\widetilde{x}_{i(Q),1}} F_\delta(1, \widetilde{x}_{2,1},\dots,\widetilde{x}_{N,1})$$
does not vanish at $Q$. Moreover, for any such integer $i(Q)$, using again that $\chi^* x_1$ vanishes at $Q$, the partial derivative~$\partial_{\widetilde{x}_{i(Q),1}} w$
also does not vanish.

Consequently, in some neighborhood of $Q$, the complex scheme $\widetilde{H} \cap V_1$ admits as coordinate system the restriction to $\widetilde{H} \cap V_1$  of:
$$(\chi^* x_1, \widetilde{x}_{2,1},\dots, {\widetilde{x}_{i(Q)-1,1}}, w,{\widetilde{x}_{i(Q)+1,1}}, \dots, \widetilde{x}_{N,1}).$$

This coordinate system is well-suited to the equation \eqref{H tilde Delta in V_1} of the divisor with normal crossings~$\widetilde{H}_\Delta \cap V_1$ in~$\widetilde{H} \cap V_1,$ and the vector bundle $\Omega^1_{\widetilde{H} \cap V_1}(\log \widetilde{H}_\Delta \cap V_1)$ on $\widetilde{H} \cap V_1$ admits the local frame near $Q$ given by the restriction to $\widetilde{H} \cap V_1$ of:
$$\big(\chi^* (dx_1/x_1), d w/w, d \widetilde{x}_{2,1},
 \dots, \\ d \widetilde{x}_{i(Q)-1,1}, d \widetilde{x}_{i(Q)+1,1},\dots,d \widetilde{x}_{N,1}\big).
 $$

Moreover the local generator $\chi_{\mid \widetilde{H} \cap V_1}^* \pi^* (dt/t)$ of the line subbundle $\chi_{\mid \widetilde{H} \cap V_1}^* \pi^* \Omega^1_C(\log \Delta)$ of the vector bundle $\Omega^1_{\widetilde{H} \cap V_1}(\log \widetilde{H}_\Delta \cap V_1)$ is given by equality \eqref{second eq comput dt sur t not H del tilde}.

Consequently the quotient coherent sheaf
$$\omega^1_{\widetilde{H}/C \mid \widetilde{H} \cap V_1} := \Omega^1_{\widetilde{H} \cap V_1}(\log \widetilde{H}_\Delta \cap V_1) \big/ \chi_{\mid \widetilde{H} \cap V_1}^* \pi^* \Omega^1_C(\log \Delta)$$
indeed admits the local frame \eqref{frame omega H tilde C} near $Q$.

\subsubsection{} The goal of this paragraph is to establish the assertion \eqref{case not H Del tilde for omega} of Lemma \ref{lem om on V_1}, namely to assume that the point $Q$ is in $E_P \setminus \widetilde{H_\Delta}$ and to prove that a local frame of the vector bundle $\omega^1_{\widetilde{H}/C \mid \widetilde{H} \cap V_1}$ near~$Q$ is given by \eqref{local frame omega Q not H Del tilde}.

Since $Q$ is not in $\widetilde{H_\Delta}$, the divisor $\widetilde{H}_\Delta \cap V_1$ in $\widetilde{H} \cap V_1$ coincides, locally near $Q$, with the divisor~$(\delta E_P) \cap V_1$ defined by the equation
$$(\chi^* x_1^{\delta})_{\mid \widetilde{H} \cap V_1} = 0,$$
and therefore, using the local coordinate system on $\widetilde{H} \cap V_1$ defined as the restriction of the family~$(\chi^* x_1, \widetilde{x}_{2,1},\dots,\widetilde{x}_{N,1})$, the vector bundle $\Omega^1_{\widetilde{H} \cap V_1}(\log \widetilde{H}_\Delta \cap V_1)$ admits the following local frame:
\begin{equation}\label{local frame Omega log Q not H Del tilde}
\big(\chi_{\mid \widetilde{H} \cap V_1}^* (d x_1/x_1), d \widetilde{x}_{2,1 \mid \widetilde{H} \cap V_1},\dots, d \widetilde{x}_{N,1 \mid \widetilde{H} \cap V_1}\big).
\end{equation}

Moreover the local generator $\chi_{\mid \widetilde{H} \cap V_1}^* \pi^* (dt/t)$ of the line subbundle $\chi_{\mid \widetilde{H} \cap V_1}^* \pi^* \Omega^1_C(\log \Delta)$ of the vector bundle $\Omega^1_{\widetilde{H} \cap V_1}(\log \widetilde{H}_\Delta \cap V_1)$ is given by equality \eqref{second eq comput dt sur t not H del tilde}.

By hypothesis, $Q$ is not in $\widetilde{H_\Delta}$, hence using the equation \eqref{H Delta tilde in V_1}, the function $w$ does not vanish at $Q$, so that, locally near $Q$, the function $w^{-1}$ is well-defined and regular. Consequently the form~$(w^{-1} d w)_{\mid \widetilde{H} \cap V_1}$ belongs to the image of the natural morphism of locally free sheaves on $\widetilde{H} \cap V_1$:
$$\Omega^1_{\widetilde{H} \cap V_1} \lra \Omega^1_{\widetilde{H} \cap V_1}(\log \widetilde{H}_\Delta \cap V_1).$$

Therefore, in the local frame \eqref{local frame Omega log Q not H Del tilde}, the coordinate of $(w^{-1} d w)_{\mid \widetilde{H} \cap V_1}$ associated to the local frame element $\chi_{\mid \widetilde{H} \cap V_1}^* (d x_1/x_1)$ of $\Omega^1_{\widetilde{H} \cap V_1}(\log \widetilde{H}_\Delta \cap V_1)$ is locally a multiple of $\chi_{\mid \widetilde{H} \cap V_1}^* x_1$, in particular it vanishes at $Q$.

Equality \eqref{second eq comput dt sur t not H del tilde} therefore implies that the coordinate of $\chi_{\mid \widetilde{H} \cap V_1}^* \pi^* (dt/t)$ associated to the local frame element $\chi_{\mid \widetilde{H} \cap V_1}^* (d x_1/x_1)$ does not vanish on some neighborhood of $Q$ in $\widetilde{H} \cap V_1$, so that the quotient coherent sheaf:
$$\omega^1_{\widetilde{H}/C \mid \widetilde{H} \cap V_1} := \Omega^1_{\widetilde{H} \cap V_1}(\log \widetilde{H}_\Delta \cap V_1) \big/ \chi_{\mid \widetilde{H} \cap V_1}^* \pi^* \Omega^1_C(\log \Delta)$$
admits the local frame \eqref{local frame omega Q not H Del tilde} near $Q$.

This concludes the proof of Lemma \ref{lem om on V_1}.

\subsection{The exceptional divisor $E'_{P'}$ and the open subsets $V'_1,\dots,V'_N$}\label{7.3}

Now let us state and prove a second lemma to simplify the local study of $\widetilde{H}'$.

\begin{lemma}\label{all in V' 1}
The following inclusion holds:
\begin{equation}\label{inc all in V' 1}
E'_{P'} \subset \bigcup_{i = 1}^N V'_i.
\end{equation}
\end{lemma}

The remainder of this subsection is devoted to the proof of Lemma \ref{all in V' 1}.

Let $Q'$ be any complex point of $E'_{P'}$. As a neighborhood of $E'_{P'} \subset Z'_{P'}$ is covered by its intersections with the open subsets $V'_0,\dots,V'_N$ of $\widetilde{X}'$, we may assume that $Q'$ is in $V'_0 \cap E'_{P'}$.

The open subset $V'_0$ of $\widetilde{X}'$ admits as coordinate system:
 $$(\chi'^*\pi'^* t', \widetilde{x}'_{1,0},\dots,\widetilde{x}'_{N,0}),$$ 
where for every integer $i\in \{1,\dots, N\}$, $\widetilde{x}'_{i,0}$ is defined by:
 $$\widetilde{x}'_{i,0} := v'_i / v'_0,$$
and satisfies the following equality of functions on $V'_0$:
\begin{equation}\label{x_i in V'_0} 
\chi'^* x'_i = \widetilde{x}'_{i,0} \,  \chi'^* \pi'^* t'.
\end{equation}

Observe moreover that the exceptional divisor $Z'_{P'} \cap V'_0$ in $V'_0$ is defined by the vanishing of~$\chi'^* \pi'^* t'$.

Let us describe the divisor $\widetilde{H}' \cap V'_0$ in $V'_0$.
Using firstly the equation \eqref{eq H' in X' F delta} of the divisor $H'$ in~$X'$ and equality \eqref{x_i in V'_0}, then the homogeneity of $F_\delta$, the divisor $\chi'^* H' \cap V'_0$ in $V'_0$ is defined by the following equation:
\begin{align*}
\chi'^* \pi'^*(u\,  t'^\delta) &= F_\delta(\widetilde{x}'_{1,0} \, \chi'^* \pi'^* t',\dots,\widetilde{x}'_{N,0} \,  \chi'^* \pi'^* t') + \chi'^* \tau^* F_{> \delta} \\
&= \chi'^* \pi'^* t'^{\delta}  \big[F_\delta(\widetilde{x}'_{1,0},\dots,\widetilde{x}'_{N,0}) + \chi'^* \pi'^* t' \, 
h'_0 \big],
\end{align*}
where $h'_0$ is  defined by:
$$h'_0 := \chi'^* \pi'^* t'^{-\delta-1} \, \, \chi'^\ast \tau^{\ast} F_{> \delta},$$ 
and, similarly to $h_1$ in Lemma \ref{lem om on V_1}, is actually a (germ of) regular function on $V'_0$ along  $Z'_{P'} \cap V'_0$, since the pullback~$\chi'^\ast \tau^* F_{> \delta}$ vanishes at order at least $\delta+ 1$ on the exceptional divisor~$Z'_{P'}$.

Consequently the divisor $\widetilde{H}' \cap V'_0 = (\chi'^* H' - Z'_{\delta}) \cap V'_0$ in $V'_0$ is defined by the following equation:
\begin{equation}\label{H' tilde in V'_0}
\chi'^* \pi'^* u = F_\delta(\widetilde{x}'_{1,0},\dots,\widetilde{x}'_{N,0}) + \chi'^* \pi'^* t' \, h'_0.
\end{equation}

Since the point $Q'$ is in $E'_{P'} \cap V'_0 = Z'_{P'} \cap \widetilde{H}' \cap V'_0$, it satisfies \eqref{H' tilde in V'_0}, and the function $\chi'^* \pi'^* t'$ vanishes at $Q'$. These observations, along with the invertibility of $u$, show that the function:
$$F_\delta(\widetilde{x}'_{1,0},\dots,\widetilde{x}'_{N,0})$$
does not vanish at $Q'$.

Since $F_\delta$ is a homogeneous polynomial of degree $\delta > 0$, there exists an integer $i$ in $\{1,\dots,N\}$ such that the function $\widetilde{x}'_{i,0}$ does not vanish at $Q'$.

By definition of $\widetilde{x}'_{i,0}$, this precisely means that $Q'$ is in the open subset $V'_i$ of $\widetilde{X}'$ defined by the non-vanishing of $v'_i$, which concludes the proof of Lemma \ref{all in V' 1}.

\subsection{Construction of $\upsilon$ and description of $\upsilon^* \omega^1_{\widetilde{H}/C}$ over $\widetilde{H}' \cap V'_1$}\label{7.4}

We now turn back to the proof of Proposition \ref{both omega iso}. As observed at the beginning of Section \ref{Proof both omega}, the assertions that we want to establish are local on $\widetilde{H}'$, so that it suffices to establish them on some neighborhood of a point $Q'$ in $E'_{P'}$.

It clearly follows from Lemma \ref{all in V' 1} that the point $Q'$ belongs to one of the open subsets $V'_1,\dots,V'_N$. Without loss of generality, we may assume that it belongs to $V'_1$.

The open subset $V'_1$ admits a local coordinate system $(\widetilde{\pi'^* t'}_1, \chi'^* x'_1, \widetilde{x}'_{2,1},\dots,\widetilde{x}'_{N,1}),$ where $\widetilde{\pi'^* t'}_1$ is defined by:
$$\widetilde{\pi'^* t'}_1 := v'_0/v'_1,$$
and satisfies the following equality of functions on $V'_1$:
\begin{equation}\label{t' in V'_1}
 \chi'^* \pi'^* t' = \widetilde{\pi'^* t'}_1 \, \chi'^* x'_1,
 \end{equation}
and where for every integer $i \in \{ 2, \dots, N\},$ $\widetilde{x}'_{i,1}$ is defined by:
$$\widetilde{x}'_{i,1} := v'_i / v'_1$$
and satisfies the following equality:
\begin{equation}\label{x_i in V'_1} 
\chi'^* x'_i = \widetilde{x}'_{i,1} \, \chi'^* x'_1.
\end{equation}

Observe moreover that the exceptional divisor $Z'_{P'} \cap V'_1$ in $V'_1$ is defined by the vanishing of $\chi'^* x'_1$.

The validity, on some neighborhood of $Q'$, of the first and second assertions of Proposition \ref{both omega iso} will clearly follow (by restriction to $\widetilde{H}' \cap V'_1$) from the following lemma, which will also be useful in establishing that the morphism of locally free sheaves $\phi$ indeed extends into an isomorphism. Note that the proof of this lemma will rely on Lemma \ref{lem om on V_1}.

\begin{lemma}\label{lem ups V' 1}
Similarly to $h_1$ in Lemma \ref{lem om on V_1}, the following expression:
$$h'_1 := \chi'^*  x'^{-\delta-1}_1 \,  \, \chi'^\ast \tau^\ast F_{> \delta}$$
defines a (germ of) regular function on $V'_1$ along $Z'_{P'} \cap V'_1$. Moreover, letting $w'_1$ be the (germ of) regular function defined as:
$$w'_1 := F_\delta(1, \widetilde{x}'_{2,1}, \dots, \widetilde{x}'_{N,1}) + \chi'^* x'_1 \, h'_1,$$
the divisor $\widetilde{H}' \cap V'_1$ in $V'_1$ is locally defined by the following equation:
\begin{equation}\label{H' tilde in V'_1 with w}
\chi'^*\pi'^*u\, \, \widetilde{\pi'^* t'}_1^{\delta} = w'_1,
\end{equation}
and the divisors $\widetilde{H}'_{\Delta'} \cap V'_1$ and $\widetilde{H'_{\Delta'}} \cap V'_1$ in $\widetilde{H}' \cap V'_1$ are respectively defined by the following equations:
\begin{equation}\label{H' tilde Del in V'_1}
\widetilde{\pi'^* t'}_{1 \mid \widetilde{H}' \cap V'_1} \, \chi'^*_{ \mid \widetilde{H}' \cap V'_1} x'_{1} = 0
\end{equation}
and
\begin{equation}\label{H' Del tilde in V'_1}
\widetilde{\pi'^* t'}_{1 \mid \widetilde{H}' \cap V'_1} = 0.
\end{equation}

Moreover the following equality of closed subschemes of $V'_1$ holds:
\begin{equation}\label{eq subschemes for ups in V' 1}
(\tau \circ \chi')^{-1}(\Sigma) \cap V'_1 = Z'_{P'} \cap V'_1,
\end{equation}
so that, using the universal property of blow-ups, there exists a unique morphism\footnote{The notation $\upsilon_{\mid V'_1}$ is slightly abusive, since there is no morphism from $\widetilde{X}'$ to $\widetilde{X}$ that would fit in the top line of the diagram \eqref{diag ups V'_1}.} $\upsilon_{\mid V'_1} : V'_1 \ra \widetilde{X}$ such that the following diagram is commutative: 
\begin{equation}\label{diag ups V'_1}
\xymatrix{
V'_1 \ar[r]^{\upsilon_{| V'_1}} \ar[d]_{\chi'_{| V'_1}} & \widetilde{X}\ar[d]^{\chi} \\
X' \ar[r]_{\tau} & X .
}
\end{equation}

Finally the vector bundle $\upsilon_{\mid \widetilde{H}' \cap V'_1}^* \omega^1_{\widetilde{H}/C}$ on $\widetilde{H}' \cap V'_1$ admits the following local description near $Q'$.
\begin{enumerate}
\item \label{case H' Del tilde for ups omega} If $Q'$ is in $\widetilde{H'_{\Delta'}}$ (or equivalently if the function $\widetilde{\pi'^* t'}_1$ vanishes at $Q'$), then there exists an integer $i(Q')$ in $\{2,\dots,N\}$ such that the partial derivative
$$\partial_{\widetilde{x}'_{i(Q'),1}} F_\delta(1, \widetilde{x}'_{2,1},\dots,\widetilde{x}'_{N,1})$$
does not vanish at $Q'$. Moreover, for any such integer $i(Q')$, the vector bundle $\upsilon_{\mid \widetilde{H}' \cap V'_1}^* \omega^1_{\widetilde{H}/C}$ admits the following local frame on some (analytic) open neighborhood of $Q'$ in $\widetilde{H}' \cap V'_1$:
\begin{multline}\label{frame ups star omega H tilde C when H' Del tilde}
\big(\chi'^*_{\mid \widetilde{H}' \cap V'_1} [dx'_1/x'_1], \big[d \widetilde{x}'_{2,1\mid \widetilde{H}' \cap V'_1}\big],\dots, \big[d \widetilde{x}'_{i(Q')-1,1\mid \widetilde{H}' \cap V'_1}\big], \\ \big[d \widetilde{x}'_{i(Q')+1,1\mid \widetilde{H}' \cap V'_1}\big],\dots,\big[d \widetilde{x}'_{N,1\mid \widetilde{H}' \cap V'_1}\big]\big).
\end{multline}

\item \label{case not H' Del tilde for ups omega} If $Q'$ is not in $\widetilde{H'_{\Delta'}}$ (or equivalently if the function $\widetilde{\pi'^* t'}_1$ does not vanish at $Q'$), then the vector bundle $\upsilon_{\mid \widetilde{H}' \cap V'_1}^* \omega^1_{\widetilde{H}/C}$ admits the following local frame on some (analytic) open neighborhood of $Q'$ in $\widetilde{H}' \cap V'_1$:
\begin{equation}\label{local frame ups star omega Q' not H' Del' tilde}
\big(\big[d \widetilde{x}'_{2,1 \mid \widetilde{H}' \cap V'_1}\big],\dots, \big[d \widetilde{x}'_{N,1 \mid \widetilde{H}' \cap V'_1}\big]\big).
\end{equation}
\end{enumerate}
\end{lemma}

The remainder of this subsection is devoted to the proof of Lemma \ref{lem ups V' 1}.

Observe first that the regularity of the (germ of) function $h'_1$ follows from the fact that $F_{> \delta}$ belongs to the ideal~$\mathfrak{m}_{X,P}^{\delta +1}$ of~$\cO_{X,P}$, and therefore the pullback $\chi'^\ast \tau^\ast F_{> \delta}$ vanishes at order at least~$\delta+ 1$ on the exceptional divisor~$Z'_{P'}$.

\subsubsection{} Let us now establish the divisor equations \eqref{H' tilde in V'_1 with w}, \eqref{H' tilde Del in V'_1}, and \eqref{H' Del tilde in V'_1}.

Using first the equation \eqref{eq H' in X' F delta} of the divisor~$H'$ in~$X'$ and equalities \eqref{t' in V'_1} and \eqref{x_i in V'_1}, then the homogeneity of $F_\delta$ and the definition of $h'_1$, one obtains that the divisor~$\chi'^* H' \cap V'_1$ in $V'_1$ is defined by the following equation:
\begin{align*}
\chi'^*\pi'^*u\, \,\widetilde{\pi'^* t'}^{\delta}_1 \,  \, \chi'^* x'^{\delta}_1 & = F_\delta(\chi'^* x'_1, \widetilde{x}'_{2,1}  \, \chi'^* x'_1, \dots,\widetilde{x}'_{N,1} \, \chi'^* x'_1)  + \chi'^* \tau^* F_{> \delta} \\
& = \chi'^* x'^{\delta}_1 \, \big[F_\delta(1,\widetilde{x}'_{2,1}, \dots,\widetilde{x}'_{N,1}) + \chi'^* x'_1 \,  h'_1 \big],
\end{align*}
and therefore that the divisor $\widetilde{H}' \cap V'_1 = (\chi'^* H' - Z'_{\delta}) \cap V'_1$ in $V'_1$ is indeed defined by the equation \eqref{H' tilde in V'_1 with w}.

The equation \eqref{H' tilde Del in V'_1} for the divisor $\widetilde{H}'_{\Delta'} \cap V'_1$ in $\widetilde{H}' \cap V'_1$ follows from its definition and from equality \eqref{t' in V'_1}, and the equation \eqref{H' Del tilde in V'_1} for the divisor $\widetilde{H'_{\Delta'}} \cap V'_1 = (\widetilde{H}'_{\Delta'} - E'_{P'}) \cap V'_1$ follows immediately from equation \eqref{H' tilde Del in V'_1}.

\subsubsection{} One easily checks that the intersection $(\tau \circ \chi')^{-1}(\Sigma) \cap V'_1$ is defined in $ V'_1$ by the equation $(\chi'^* x'_1 = 0),$ so that it coincides with the exceptional divisor $Z'_{P'} \cap V'_1$ as a subscheme of $V'_1$. This establishes equality \eqref{eq subschemes for ups in V' 1} and the existence of a (necessarily unique) morphism $\upsilon_{\mid V'_1} : V'_1 \ra \widetilde{X}$ such that the diagram \eqref{diag ups V'_1} is commutative.

In the remainder of this paragraph, we give an explicit description of this morphism.

Using equalities \eqref{t' in V'_1} and \eqref{x_i in V'_1}, one obtains that the restriction to $V_1'$ of the morphism $\tau \circ \chi'$ from $\widetilde{X}'$ to $X$ satisfies the following equalities of functions on $V'_1$:
$$(\tau \circ \chi')^* \pi^* t = \chi'^* \pi'^*(u \,  t'^{\delta}) = \chi'^* \pi'^*u \,\, \widetilde{\pi'^* t'}^{\delta} _1 \,\, \chi'^* x'^{\delta}_1,$$
$$(\tau \circ \chi')^* x_1 = \chi'^* x'_1,$$
and for every integer $i \in \{2, \dots, N\}$:
$$(\tau \circ \chi')^* x_i = \widetilde{x}'_{i,1}\, \, \chi'^* x'_1.$$

Consequently
 $\upsilon_{\mid V'_1}$ satisfies the following equality 
of morphisms from $V'_1$ to $\PP^N$:
$$\upsilon_{\mid  V'_1}^* [v_0 : \dots : v_N] = [\chi'^*\pi'^*u\, \,  \widetilde{\pi'^* t'}^{\delta}_1 \,  \, \chi'^* x'^{\delta-1}_1 : 1 : \widetilde{x}'_{2,1} : \dots : \widetilde{x}'_{N,1}].$$

In particular, this morphism has values in the open subset $V_1 := (v_1 \neq 0)$ in $\widetilde{X}$. 

In the coordinate system $(\widetilde{\pi^* t}_1, \chi^* x_1, \widetilde{x}_{2,1},\dots,\widetilde{x}_{N,1})$ of $V_1$ defined at the beginning of Subsection~\ref{subsec desc omega H tilde C V_1} above, the morphism $$\upsilon_{\mid V'_1} :  V'_1 \lra V_1$$ may be described by the following equalities:
\begin{equation}\label{ups on V'_1 t bis}
 \upsilon_{\mid V'_1}^* \widetilde{\pi^* t}_1 =\chi'^*\pi'^*u\, \,  \widetilde{\pi'^* t'}^{\delta}_1 \, \, \chi'^* x'^{\delta-1}_1,
 \end{equation}
\begin{equation}\label{ups on V'_1 x_1 bis} 
\upsilon_{\mid  V'_1}^* \chi^* x_1 = \chi'^* x'_1,
\end{equation}
and for every integer $i\in \{2, \dots, N\}$: 
\begin{equation}\label{ups on V'_1 x_i bis}
 \upsilon_{\mid  V'_1}^* \widetilde{x}_{i,1} = \widetilde{x}'_{i,1}.
 \end{equation}

Observe that equalities \eqref{ups on V'_1 t bis}, \eqref{ups on V'_1 x_1 bis}, and \eqref{ups on V'_1 x_i bis} imply by restriction the following equalities of functions on $\widetilde{H}' \cap V'_1$: 
\begin{equation}\label{ups on V'_1 t}
 \upsilon_{\mid \widetilde{H}'\cap V'_1}^* \widetilde{\pi^* t}_1 =\big(\chi'^*\pi'^*u\, \,  \widetilde{\pi'^* t'}^{\delta}_1 \,  \, \chi'^* x'^{\delta-1}_1\big)_{\mid \widetilde{H}' \cap V'_1},
 \end{equation}
\begin{equation}\label{ups on V'_1 x_1} 
\upsilon_{\mid  \widetilde{H}'\cap V'_1}^* \chi^* x_1 = (\chi'^* x'_1)_{\mid \widetilde{H}' \cap V'_1},
\end{equation}
and for every integer $i\in \{2, \dots, N\}$: 
\begin{equation}\label{ups on V'_1 x_i}
 \upsilon_{\mid  \widetilde{H}'\cap V'_1}^* \widetilde{x}_{i,1} = \widetilde{x}'_{i,1\mid \widetilde{H}' \cap V'_1}.
 \end{equation}

\subsubsection{} In order to deduce the assertions \eqref{case H' Del tilde for ups omega} and \eqref{case not H' Del tilde for ups omega} of Lemma \ref{lem ups V' 1} from the assertions \eqref{case H Del tilde for omega} and \eqref{case not H Del tilde for omega} of Lemma \ref{lem om on V_1}, we will need the following equality of closed subsets of $\widetilde{H}' \cap V'_1$:
\begin{equation}\label{eq supports H Del tilde ups}
\big| \widetilde{H'_{\Delta'}} \cap V'_1 \big| = \big| \upsilon_{\mid \widetilde{H}' \cap V'_1}^* (\widetilde{H_\Delta} \cap V_1) \big|.
\end{equation}

In the remainder of this paragraph, we establish this equality.

Using successively the definition of the function $h_1$ on $V_1$ from Lemma \ref{lem om on V_1}, the commutativity of the diagram \eqref{diag ups V'_1}, and the definition of the function~$h'_1$ on $V'_1$, one easily obtains the following equalities of functions on $V'_1$:
\begin{align}
\upsilon_{| V'_1}^* h_1 & = \upsilon _{| V'_1}^* (\chi^* x_1^{-\delta-1} \, \chi^* F_{> \delta}) \nonumber  \\
& = \chi'^* (x'^{-\delta-1}_1 \, \tau^* F_{> \delta} ) \nonumber \\
\label{h_1 on V'_1} &= h'_1.
\end{align}

Consequently, using first the definition of the function $w$ on $V_1$ from Lemma \ref{lem om on V_1}, then equalities \eqref{ups on V'_1 x_1 bis}, \eqref{ups on V'_1 x_i bis} and \eqref{h_1 on V'_1}, and finally the definition of the function $w'_1$ on $V'_1$, one obtains the following equalities of functions on $V'_1$:
\begin{align}
\upsilon_{| V'_1}^* w & = \upsilon_{| V'_1}^* F_\delta(1, \widetilde{x}_{2,1},\dots,\widetilde{x}_{N,1}) + \upsilon_{| V'_1}^* (\chi^* x_1 \, h_1) \nonumber \\
& = F_\delta(1, \widetilde{x}'_{2,1}, \dots, \widetilde{x}'_{N,1}) + \chi'^* x'_1 \,  h'_1 \nonumber \\
\label{both w on V'_1} & = w'_1.
\end{align}

It therefore follows from the equation \eqref{H Delta tilde in V_1} of the divisor $\widetilde{H_\Delta} \cap V_1$ in $\widetilde{H} \cap V_1$ and from equality \eqref{both w on V'_1} that the divisor $\upsilon_{\mid \widetilde{H}' \cap V'_1}^* (\widetilde{H_\Delta} \cap V_1)$ in $\widetilde{H}' \cap V'_1$ is defined by the following equation:
\begin{equation}\label{eq ups star H Del tilde in V'_1}
w'_{1 \mid \widetilde{H}' \cap V'_1} = 0.
\end{equation}

Since equality \eqref{H' tilde in V'_1 with w} is satisfied on $\widetilde{H}' \cap V'_1$ and $u$ is invertible, the vanishing of $w'_1$ at any given complex point of $\widetilde{H}' \cap V'_1$ is equivalent to the vanishing of $\widetilde{\pi'^* t'}_1$ at this same point.

Using the equations \eqref{H' Del tilde in V'_1} and \eqref{eq ups star H Del tilde in V'_1} for the two divisors $\widetilde{H'_{\Delta'}} \cap V'_1$ and $\upsilon_{\mid \widetilde{H}' \cap V'_1}^* (\widetilde{H_\Delta} \cap V_1)$ in~$\widetilde{H}' \cap V'_1$ respectively, this equivalence becomes equality \eqref{eq supports H Del tilde ups}.

\subsubsection{} In this paragraph, we establish the assertion \eqref{case H' Del tilde for ups omega} of Lemma \ref{lem ups V' 1}, namely we assume that the point $Q'$ is in $\widetilde{H'_{\Delta'}}$ and prove that there exists an integer $i(Q')$ in $\{2,\dots,N\}$ such that the partial derivative
$$\partial_{\widetilde{x}'_{i(Q'),1}} F_\delta(1, \widetilde{x}'_{2,1},\dots,\widetilde{x}'_{N,1})$$
does not vanish at $Q'$, and that, for any such integer, a local frame of the vector bundle $\upsilon_{\mid \widetilde{H}' \cap V'_1}^* \omega^1_{\widetilde{H}/C}$ near $Q'$ is given by \eqref{frame ups star omega H tilde C when H' Del tilde}.

Using the equation \eqref{H' Del tilde in V'_1} for the divisor $\widetilde{H'_{\Delta'}} \cap V'_1$ in $\widetilde{H}' \cap V'_1$, the function $\widetilde{\pi'^* t'}_1$ vanishes at~$Q'$. Consequently, using that equality \eqref{H' tilde in V'_1 with w} holds on $\widetilde{H}' \cap V'_1$, the function $w'_1$ also vanishes at $Q'$. Moreover, since the point $Q'$ is in $Z'_{P'}$, the function~$\chi'^* x'_1$ vanishes at $Q'$. Hence, by definition of the function $w'_1$, the function
$$F_\delta(1, \widetilde{x}'_{2,1},\dots,\widetilde{x}'_{N,1})$$
vanishes at $Q'$.

Since the homogeneous polynomial $F_\delta$ has non-zero discriminant, there exists an integer $i(Q')$ in $\{2,\dots,N\}$ such that the partial derivative:
$$\partial_{\widetilde{x}'_{i(Q'),1}} F_\delta(1, \widetilde{x}'_{2,1}, \dots, \widetilde{x}'_{N,1})$$
does not vanish at $Q'$.

It follows from equality \eqref{eq supports H Del tilde ups} that, since the point $Q'$ is in $\widetilde{H'_{\Delta'}}$, its image
$$Q := \upsilon(Q')$$
is in $\widetilde{H_\Delta} \cap V_1$. In order to find a local frame for the vector bundle $\upsilon_{\mid \widetilde{H}' \cap V'_1}^* \omega^1_{\widetilde{H}/C}$, we therefore need to apply the assertion \eqref{case H Del tilde for omega} of Lemma \ref{lem om on V_1} to the point $Q$. However, with the notation of that assertion, we first need to find a suitable integer $i(Q)$.

To achieve this, let us first show that, for any integer $i(Q')$ as above, the partial derivative~$\partial_{\widetilde{x}_{i(Q'),1}} w$ does not vanish at $Q$. Using first equality \eqref{both w on V'_1}, then the chain rule applied to the coordinate system $$(\widetilde{\pi^* t}_1, \chi^* x_1, \widetilde{x}_{2,1},\dots,\widetilde{x}_{N,1})$$ of $V_1,$ and then equalities \eqref{ups on V'_1 t bis}, \eqref{ups on V'_1 x_1 bis}, and \eqref{ups on V'_1 x_i bis}, we obtain the following expression for the partial derivative $\partial_{\widetilde{x}'_{i(Q'), 1}} w'_1$ in the local coordinate system $(\widetilde{\pi'^* t'}_1, \chi'^* x'_1, \widetilde{x}'_{2,1},\dots,\widetilde{x}'_{N,1})$ of $V'_1$: 
\begin{align*}
\partial_{\widetilde{x}'_{i(Q'),1}} w'_1
 & = \partial_{\widetilde{x}'_{i(Q'),1}} (\upsilon_{| V'_1}^* w) \\
& = \partial_{\widetilde{x}'_{i(Q'),1}} (\upsilon_{| V'_1}^* \widetilde{\pi^* t}_1) \, \upsilon_{| V'_1}^* \partial_{\widetilde{\pi^* t}_1} w
 + \partial_{\widetilde{x}'_{i(Q'),1}} (\upsilon_{| V'_1}^* \chi^* x_1) \, \upsilon_{| V'_1}^* \partial_{\chi^* x_1} w \\
 & \quad + \sum_{j=2}^{N} \partial_{\widetilde{x}'_{i(Q'),1}} (\upsilon_{| V'_1}^* \widetilde{x}_{j,1}) \, \upsilon_{| V'_1}^* \partial_{\widetilde{x}_{j,1}} w\\
&= \partial_{\widetilde{x}'_{i(Q'),1}} (\widetilde{x}'_{i(Q'),1}) \, \upsilon^*_{| V'_1} \partial_{\widetilde{x}_{i(Q'),1}} w \\
& = \upsilon_{| V'_1}^* \partial_{\widetilde{x}_{i(Q'),1}} w.
\end{align*}

Consequently, since $\partial_{\widetilde{x}'_{i(Q'),1}} w'_1$ does not vanish at $Q'$, we obtain that $\partial_{\widetilde{x}_{i(Q'),1}} w$ does not vanish at $Q$. 

Moreover, using the commutativity of the diagram \eqref{diag ups V'_1}, the point $Q = \upsilon(Q')$ is in the exceptional divisor $Z_P$, or equivalently, the function $\chi^* x_1$ vanishes at $Q$.

Consequently, by definition of the function $w$ on $V_1$, the partial derivative:
$$\partial_{\widetilde{x}_{i(Q'),1}} F_\delta(1, \widetilde{x}_{2,1},\dots,\widetilde{x}_{N,1})$$
does not vanish at $Q$.

We may therefore apply the assertion \eqref{case H Del tilde for omega} of Lemma \ref{lem om on V_1} with the integer:
$$i(Q) := i(Q'),$$
to obtain that a local frame of the vector bundle $\upsilon_{\mid \widetilde{H}' \cap V'_1}^* \omega^1_{\widetilde{H}/C}$ near $Q'$ is given by the pullback by $\upsilon$ of the local frame \eqref{frame omega H tilde C} of the vector bundle $\omega^1_{\widetilde{H}/C \mid \widetilde{H} \cap V_1}$ near $Q$. It immediately follows from equalities \eqref{ups on V'_1 x_1} and \eqref{ups on V'_1 x_i} that this pullback is precisely given by \eqref{frame ups star omega H tilde C when H' Del tilde}.

This concludes the proof of the assertion \eqref{case H' Del tilde for ups omega} of Lemma \ref{lem ups V' 1}.

\subsubsection{} In this paragraph, we conclude the proof of Lemma \ref{lem ups V' 1} by establishing the assertion \eqref{case not H' Del tilde for ups omega}, namely by assuming that the point $Q'$ is in $E'_{P'} \setminus \widetilde{H'_{\Delta'}}$ and by proving that a local frame of the vector bundle $\upsilon_{\mid \widetilde{H}' \cap V'_1}^* \omega^1_{\widetilde{H}/C}$ near $Q'$ is given by \eqref{local frame ups star omega Q' not H' Del' tilde}.

It follows from equality \eqref{eq supports H Del tilde ups} that, since the point $Q'$ is not in $\widetilde{H'_{\Delta'}}$, its image
$$Q := \upsilon(Q')$$
is not in $\widetilde{H_\Delta} \cap V_1$. Consequently we may apply the assertion \eqref{case not H Del tilde for omega} of Lemma \ref{lem om on V_1} to obtain that a local frame of the vector bundle $\upsilon_{\mid \widetilde{H}' \cap V'_1}^* \omega^1_{\widetilde{H}/C}$ near $Q'$ is given by the pullback by $\upsilon$ of the local frame \eqref{local frame omega Q not H Del tilde} of the vector bundle $\omega^1_{\widetilde{H}/C \mid \widetilde{H} \cap V_1}$ near $Q$. It immediately follows from equality \eqref{ups on V'_1 x_i} that this pullback is precisely given by \eqref{local frame ups star omega Q' not H' Del' tilde}, which establishes the assertion \eqref{case not H' Del tilde for ups omega} of Lemma \ref{lem ups V' 1}.

This concludes the proof of Lemma \ref{lem ups V' 1}.

\subsection{Completing the proof of Proposition \ref{both omega iso}}\label{7.5}

In this subsection, we use Lemma \ref{lem ups V' 1} to conclude the proof of Proposition \ref{both omega iso}.

Let us first mention the following equality of meromorphic differential forms on $\widetilde{H}' \cap V'_1$, which follows from equality \eqref{t' in V'_1}. 
\begin{equation}\label{first eq comput dt sur t not H del tilde prime}
\chi'^*_{\mid \widetilde{H}' \cap V'_1} \pi'^* (dt'/t')
= (d \widetilde{\pi'^* t'}_1 / \widetilde{\pi'^* t'}_1)_{\mid \widetilde{H}' \cap V'_1} + \chi'^*_{\mid \widetilde{H}' \cap V'_1} (d x'_1/x'_1).
\end{equation}

The remainder of this proof is split into two cases, depending on whether the point $Q'$ is in $\widetilde{H'_{\Delta'}}$. In both cases, we exhibit a local frame of the vector bundle $\omega^1_{\widetilde{H}'/C' \mid \widetilde{H}' \cap V'_1}$ near $Q'$ and observe that it coincides with the local frame of the vector bundle $\upsilon_{\mid \widetilde{H}' \cap V'_1}^* \omega^1_{\widetilde{H}/C}$ given by \eqref{frame ups star omega H tilde C when H' Del tilde} (respectively by \eqref{local frame ups star omega Q' not H' Del' tilde}), thus proving that, locally near $Q'$, the morphism $\phi$ extends to an isomorphism~$\widetilde{\phi}$ between both vector bundles.

\subsubsection{} In this paragraph, we consider the case where $Q'$ is in $\widetilde{H'_{\Delta'}}$. It follows from the assertion \eqref{case H' Del tilde for ups omega} of Lemma \ref{lem ups V' 1} that there exists an integer $i(Q')$ in $\{2,\dots,N\}$ such that the partial derivative:
$$\partial_{\widetilde{x}'_{i(Q'),1}} F_\delta(1, \widetilde{x}'_{2,1},\dots,\widetilde{x}'_{N,1})$$
does not vanish at $Q'$.

Without loss of generality, we may assume that $i(Q') = 2$. Moreover, using that the function~$\chi'^* x'_1$ vanishes at $Q'$ because $Q'$ is in $E'_{P'}$, the partial derivative
$\partial_{\widetilde{x}'_{2,1}} w'_1$
also does not vanish.

Consequently some neighborhood of $Q'$ in $V'_1$ admits as coordinate system:
$$(\widetilde{\pi'^* t'}_1, \chi'^* x'_1, w'_1, \widetilde{x}'_{3,1},\dots,\widetilde{x}'_{N,1}).$$

Using the equation \eqref{H' tilde in V'_1 with w} for the divisor $\widetilde{H}' \cap V'_1$ in $V'_1$, we obtain that some neighborhood of $Q'$ in this divisor admits as coordinate system the restriction to $\widetilde{H}' \cap V'_1$ of:
$$(\widetilde{\pi'^* t'}_1, \chi'^* x'_1, \widetilde{x}'_{3,1},\dots,\widetilde{x}'_{N,1}).$$

Moreover, using the equation \eqref{H' tilde Del in V'_1} for the divisor with normal crossings $\widetilde{H}'_{\Delta'} \cap V'_1$ in $\widetilde{H}' \cap V'_1$, we obtain that, on some neighborhood of $Q'$ in $\widetilde{H}' \cap V'_1$, the vector bundle $\Omega^1_{\widetilde{H}' \cap V'_1}(\log \widetilde{H}'_{\Delta'} \cap V'_1)$ admits the following local frame:
$$\big((d \widetilde{\pi'^* t'}_1 / \widetilde{\pi'^* t'}_1)_{ \mid \widetilde{H}' \cap V'_1}, \chi'^*_{ \mid \widetilde{H}' \cap V'_1} (d x'_1/x'_1), d \widetilde{x}'_{3,1 \mid \widetilde{H}' \cap V'_1},\dots,d \widetilde{x}'_{N,1  \mid \widetilde{H}' \cap V'_1}\big).$$

Moreover the local generator $\chi'^*_{\mid \widetilde{H}' \cap V'_1} \pi'^* (dt'/t')$ of the line subbundle $\chi'^*_{\mid \widetilde{H}' \cap V'_1} \pi'^* \Omega^1_{C'}(\log \Delta')$ of the vector bundle $\Omega^1_{\widetilde{H}' \cap V'_1}(\log \widetilde{H}'_{\Delta'} \cap V'_1)$ is given by equality \eqref{first eq comput dt sur t not H del tilde prime}.

Consequently the quotient coherent sheaf
$$\omega^1_{\widetilde{H}'/C' \mid \widetilde{H}' \cap V'_1} := \Omega^1_{\widetilde{H}' \cap V'_1}(\log \widetilde{H}'_{\Delta'} \cap V'_1) \big/ \chi'^*_{\mid \widetilde{H}' \cap V'_1} \pi'^* \Omega^1_{C'}(\log \Delta')$$
admits the following local frame near $Q'$:
\begin{equation}\label{frame omega H' tilde C' V'_1}
 \big(\chi'^*_{\mid \widetilde{H}' \cap V'_1} [d x'_1/x'_1], \big[d \widetilde{x}'_{3,1\mid \widetilde{H}' \cap V'_1}\big],\dots, \big[d \widetilde{x}'_{N,1\mid \widetilde{H}' \cap V'_1 }\big]\big).
 \end{equation}

As the local frames \eqref{frame ups star omega H tilde C when H' Del tilde} (with $i(Q') = 2$) and \eqref{frame omega H' tilde C' V'_1} coincide, we obtain that the morphism $\phi$ indeed extends to an isomorphism~$\widetilde{\phi}$ of vector bundles on some neighborhood of $Q'$, thus proving Proposition \ref{both omega iso} in the case where $Q'$ is in $\widetilde{H'_{\Delta'}}$.

\subsubsection{} In this paragraph, we consider the remaining case where $Q'$ is not in $\widetilde{H'_{\Delta'}}$.

Using the equation \eqref{H' Del tilde in V'_1} of the divisor $\widetilde{H'_{\Delta'}} \cap V'_1$ in $\widetilde{H}' \cap V'_1$, we obtain that the function $\widetilde{\pi'^* t'}_{1 \mid \widetilde{H}' \cap V'_1}$ does not vanish at $Q'$. Using moreover that $Q'$ is in the divisor $E'_{P'}$, we obtain that the function~$\chi'^* x'_1$ vanishes at $Q'$.

Let us compute, using this information, the partial derivative at $Q'$ of the equation \eqref{H' tilde in V'_1 with w} for the divisor $\widetilde{H}' \cap V'_1$ in $V'_1$ relatively to the coordinate $\widetilde{\pi'^* t'}_1$:
\begin{align}
& \partial_{\widetilde{\pi'^* t'}_1}\big[\chi'^*\pi'^*u \, \, \widetilde{\pi'^* t'}_1^{\delta} - w'_1\big]_{\mid Q'} \nonumber \\
& \qquad = \partial_{\widetilde{\pi'^* t'}_1}\big[\chi'^*\pi'^*u \, \, \widetilde{\pi'^* t'}_1^{\delta} - F_\delta(1,\widetilde{x}'_{2,1}, \dots,\widetilde{x}'_{N,1}) - \chi'^* x'_1 \, h'_1\big]_{\mid Q'} \label{zeroth eq comput partial t' H' tilde} \\
& \qquad = \big[(\partial_{\widetilde{\pi'^* t'}_1} \chi'^* \pi'^* u) \, \widetilde{\pi'^* t'}_1^{\delta} + \delta \, (\chi'^* \pi'^* u) \, \widetilde{\pi'^* t'}_1^{\delta-1} - \chi'^* x'_1 \, \partial_{\widetilde{\pi'^* t'}_1} h'_1\big]_{\mid Q'} \nonumber \\
& \qquad = \big[(\chi'^* x'_1) \, (\chi'^* \pi'^* \partial_{t'} u) \, \widetilde{\pi'^* t'}_1^{\delta} + \delta \, (\chi'^* \pi'^* u) \, \widetilde{\pi'^* t'}_1^{\delta-1} - \chi'^* x'_1 \, \partial_{\widetilde{\pi'^* t'}_1} h'_1\big]_{\mid Q'} \label{second eq comput partial t' H' tilde} \\
& \qquad = \delta \, \big[(\chi'^* \pi'^* u) \, \widetilde{\pi'^* t'}_1^{\delta-1}\big]_{\mid Q'} \label{third eq comput partial t' H' tilde} \\
& \qquad \neq 0, \label{fourth eq comput partial t' H' tilde}
\end{align}
where in \eqref{zeroth eq comput partial t' H' tilde} we have used the definition of the function $w'_1$, in \eqref{second eq comput partial t' H' tilde} we have used equality \eqref{t' in V'_1} and the chain rule, in \eqref{third eq comput partial t' H' tilde} we have used the vanishing of $\chi'^* x'_1$ at $Q'$, and in \eqref{fourth eq comput partial t' H' tilde} we have used the invertibility of $u$ and the non-vanishing of $\widetilde{\pi'^* t'}_1$ at $Q'$.

Consequently this partial derivative does not vanish locally near $Q'$, and some neighborhood of~$Q'$ in the divisor $\widetilde{H}' \cap V'_1$ in $V'_1$ admits the following local coordinate system:
$$\big(\chi'^*_{\mid \widetilde{H}' \cap V'_1} x'_1, \widetilde{x}'_{2,1 \mid \widetilde{H}' \cap V'_1},\dots, \widetilde{x}'_{N,1 \mid \widetilde{H}' \cap V'_1}\big).$$

By hypothesis, locally near $Q'$, the divisor $\widetilde{H}'_{\Delta'} \cap V'_1 = (\widetilde{H'_{\Delta'}} + E'_{P'}) \cap V'_1$ in $\widetilde{H}' \cap V'_1$ coincides with the divisor $E'_{P'} \cap V'_1$, defined by the vanishing of $\chi'^*_{\mid \widetilde{H}' \cap V'_1} x'_1$.

Therefore the vector bundle $\Omega^1_{\widetilde{H}' \cap V'_1}(\log \widetilde{H}'_{\Delta'} \cap V'_1)$ admits the following local frame near $Q'$:
\begin{equation}\label{local frame Omega log Q' not H' Del' tilde}
\big(\chi'^*_{\mid \widetilde{H}' \cap V'_1} (d x'_1 / x'_1), d \widetilde{x}'_{2,1 \mid \widetilde{H}' \cap V'_1},\dots, d \widetilde{x}'_{N,1 \mid \widetilde{H}' \cap V'_1}\big).
\end{equation}

Moreover the local generator $\chi'^*_{\mid \widetilde{H}' \cap V'_1} \pi'^* (dt'/t')$ of the line subbundle $\chi'^*_{\mid \widetilde{H}' \cap V'_1} \pi'^* \Omega^1_{C'}(\log \Delta')$ of the vector bundle $\Omega^1_{\widetilde{H}' \cap V'_1}(\log \widetilde{H}'_{\Delta'} \cap V'_1)$ is given by equality \eqref{first eq comput dt sur t not H del tilde prime}.

As we have already mentioned, by hypothesis, the function $\widetilde{\pi'^* t'}_1$ does not vanish at $Q'$, so that, locally near $Q'$, the function $\widetilde{\pi'^* t'}_1^{-1}$ is well-defined and regular. Consequently the form $(d \widetilde{\pi'^* t'}_1 / \widetilde{\pi'^* t'}_1)_{\mid \widetilde{H}' \cap V'_1}$ belongs to the image of the natural morphism of locally free sheaves on $\widetilde{H}' \cap V'_1$:
$$\Omega^1_{\widetilde{H}' \cap V'_1} \lra \Omega^1_{\widetilde{H}' \cap V'_1}(\log \widetilde{H}'_{\Delta'} \cap V'_1).$$

Therefore, in the local frame \eqref{local frame Omega log Q' not H' Del' tilde}, the coordinate of $(d \widetilde{\pi'^* t'}_1 / \widetilde{\pi'^* t'}_1)_{\mid \widetilde{H}' \cap V'_1}$ associated to the local frame element $\chi'^*_{\mid \widetilde{H}' \cap V'_1} (d x'_1/x'_1)$ of $\Omega^1_{\widetilde{H}' \cap V'_1}(\log \widetilde{H}'_{\Delta'} \cap V'_1)$ is locally a multiple of $\chi'^*_{\mid \widetilde{H}' \cap V'_1} x'_1$, in particular it vanishes at $Q'$.

Equality \eqref{first eq comput dt sur t not H del tilde prime} therefore implies that the coordinate of $\chi'^*_{\mid \widetilde{H}' \cap V'_1} \pi'^* (dt'/t')$ associated to the local frame element $\chi'^*_{\mid \widetilde{H}' \cap V'_1} (d x'_1/x'_1)$ does not vanish on some neighborhood of $Q'$ in $\widetilde{H}' \cap V'_1$, so that the quotient coherent sheaf:
$$\omega^1_{\widetilde{H}'/C' \mid \widetilde{H}' \cap V'_1} := \Omega^1_{\widetilde{H}' \cap V'_1}(\log \widetilde{H}'_{\Delta'} \cap V'_1) \big/ \chi'^*_{\mid \widetilde{H}' \cap V'_1} \pi'^* \Omega^1_{C'}(\log \Delta')$$
admits the following local frame\footnote{This local frame actually shows that the natural morphism of coherent sheaves from $\Omega^1_{\widetilde{H}'/C'}$ to $\omega^1_{\widetilde{H}'/C'}$ induces by restriction an isomorphism on some neighborhood of $Q'$ in $\widetilde{H}' \cap V'_1$.} near $Q'$:
\begin{equation}\label{local frame omega Q' not H' Del' tilde}
\big(\big[d \widetilde{x}'_{2,1 \mid \widetilde{H}' \cap V'_1}\big],\dots, \big[d \widetilde{x}'_{N,1 \mid \widetilde{H}' \cap V'_1}\big]\big).
\end{equation}

As the local frames \eqref{local frame ups star omega Q' not H' Del' tilde} and \eqref{local frame omega Q' not H' Del' tilde} coincide, we obtain that the morphism $\phi$ indeed extends to an isomorphism~$\widetilde{\phi}$ of vector bundles on some neighborhood of $Q'$, thus proving Proposition \ref{both omega iso} in the case where $Q'$ is not in $\widetilde{H'_{\Delta'}}$.

This concludes the proof of Proposition \ref{both omega iso}.

\section*{Declarations}

\noindent \textbf{Competing interests} The author declares that he has no competing interests.

\noindent \textbf{Data availability} Data sharing is not applicable to this article as no datasets were generated or analyzed during the research.


\begin{thebibliography}{AGZV85}

\bibitem[AGZV85]{AGZV85}
V.~I. Arnold, S.~M. Guse\u{\i}n-Zade, and A.~N. Varchenko.
\newblock {\em Singularities of differentiable maps. {V}ol. {I}}, volume~82 of
  {\em Monographs in Mathematics}.
\newblock Birkh\"{a}user Boston, Inc., Boston, MA, 1985.

\bibitem[Del70]{Deligne70}
P.~Deligne.
\newblock {\em \'{E}quations diff\'{e}rentielles \`a points singuliers
  r\'{e}guliers}.
\newblock Lecture Notes in Mathematics, Vol. 163. Springer-Verlag, Berlin-New
  York, 1970.

\bibitem[Ful98]{Fulton98}
W.~Fulton.
\newblock {\em Intersection theory}, volume~2 of {\em Ergebnisse der Mathematik
  und ihrer Grenzgebiete (3)}.
\newblock Springer-Verlag, Berlin, second edition, 1998.

\bibitem[Kat18]{Kato18}
K.~Kato.
\newblock Height functions for motives.
\newblock {\em Selecta Math. (N.S.)}, 24(1):403--472, 2018.


\bibitem[MO70]{Milnor-Orlik70}
J.~Milnor and P.~Orlik.
\newblock Isolated singularities defined by weighted homogeneous polynomials.
\newblock {\em Topology}, 9:385--393, 1970.

\bibitem[Mor24]{Mordant22}
T.~Mordant.
\newblock Griffiths heights and pencils of hypersurfaces.
\newblock \texttt{https://arxiv.org/abs/2212.11019v3}, December 2024, to appear in
  \emph{M{\'e}m. Soc. Math. Fr.}

\bibitem[Mor25]{MordantGIT1}
T.~Mordant.
\newblock Pencils of projective hypersurfaces, {G}riffiths heights and
  geometric invariant theory~{I}.
\newblock \texttt{http://arxiv.org/abs/2506.15334}, 2025.

\end{thebibliography}
\end{document}